\documentclass[12pt]{amsart}
\usepackage[margin=1.2in]{geometry}
\usepackage{graphicx,latexsym}
\usepackage{tikz}
\usepackage{amsfonts, amssymb, amsmath, amsthm, bm, hyperref,verbatim}
\usepackage{tikz}
\usetikzlibrary{matrix,arrows}
\usepackage{mathrsfs}
\allowdisplaybreaks

\newcommand{\ds}{\displaystyle}
\newcommand{\N}{\mathbb{N}}
\newcommand{\Z}{\mathbb{Z}}
\newcommand{\R}{\mathbb{R}}

\newcommand{\ZZ}{\mathcal{Z}}
\newcommand{\A}{\mathcal{A}}

\newcommand{\W}{\mathscr{W}}
\newcommand{\V}{\mathscr{V}}

\newcommand{\dx}{{\rm d}x }

\newcommand{\dt}{{\rm d}t }
\newcommand{\dxi}{{\rm d}\xi }

\newcommand{\beq}{\begin{eqnarray}}
\newcommand{\eeq}{\end{eqnarray}}

\newcommand{\beqs}{\begin{eqnarray*}}
\newcommand{\eeqs}{\end{eqnarray*}}

\newtheorem{theorem}{Theorem}[section]
\newtheorem{proposition}[theorem]{Proposition}
\newtheorem{lemma}[theorem]{Lemma}

\theoremstyle{definition}
\newtheorem{definition}[theorem]{Definition}
\newtheorem{example}[theorem]{Example}
\newtheorem{problem}[theorem]{Problem}

\theoremstyle{remark}
\newtheorem{remark}[theorem]{Remark}

\newtheorem{notation}[theorem]{Notation}

\numberwithin{equation}{section}

\usetikzlibrary{arrows,chains,matrix,positioning,scopes}
\makeatletter
\tikzset{join/.code=\tikzset{after node path={%
\ifx\tikzchainprevious\pgfutil@empty\else(\tikzchainprevious)%
edge[every join]#1(\tikzchaincurrent)\fi}}}
\makeatother
\tikzset{>=stealth',every on chain/.append style={join},
         every join/.style={->}}
\tikzstyle{labeled}=[execute at begin node=$\scriptstyle,
   execute at end node=$]
\newcommand\condwM{\operatorname{wM}}
\newcommand\condM{\operatorname{M}}
\newcommand\condN{\operatorname{N}}

\newcommand\condQ{(Q)}
\newcommand\condwQ{(wQ)}
\newcommand\condDN{\operatorname{DN}}

\newcommand\condooOmega{\overline{\overline{\Omega}}}
\newcommand\condSquare{\operatorname{S}}

\newcommand\norm[1]{\left\lVert#1\right\rVert}

\DeclareMathOperator{\id}{id}

\begin{document}
\title[Weighted $(PLB)$-spaces of ultradifferentiable functions]{Weighted $(PLB)$-spaces of ultradifferentiable functions and multiplier spaces}

\author[A. Debrouwere]{Andreas Debrouwere}
\thanks{A. Debrouwere was supported by  FWO-Vlaanderen through the postdoctoral grant 12T0519N}
\address{Department of Mathematics and Data Science \\ Vrije Universiteit Brussel, Belgium\\ Pleinlaan 2 \\ 1050 Brussels \\ Belgium}
\email{andreas\_debrouwere@hotmail.com}

\author[L. Neyt]{Lenny Neyt}
\thanks{L. Neyt gratefully acknowledges support by FWO-Vlaanderen through the postdoctoral grant 12ZG921N}
\address{Department of Mathematics: Analysis, Logic and Discrete Mathematics\\ Ghent University\\ Krijgslaan 281\\ 9000 Gent\\ Belgium}
\email{lenny.neyt@UGent.be}


\begin{abstract}
We study weighted $(PLB)$-spaces of ultradifferentiable functions  defined via a weight function (in the sense of Braun, Meise and Taylor) and a weight system. We characterize when such spaces are ultrabornological in terms of the defining weight system. This generalizes Grothendieck's classical result that the space $\mathcal{O}_M$ of slowly increasing smooth functions is ultrabornological to the context of ultradifferentiable functions. Furthermore, we determine the multiplier spaces of Gelfand-Shilov spaces and, by using the above result, characterize when such spaces are ultrabornological. In particular, we show that the multiplier space of the space of Fourier ultrahyperfunctions is ultrabornological, whereas the one of the space of Fourier hyperfunctions is not.

\end{abstract}

\maketitle

\section{Introduction}

Countable projective limits of countable inductive limits of Banach spaces, called $(PLB)$-spaces, arise naturally in  functional analysis. Classical examples are the space of distributions, the space of real analytic functions and the space $\mathcal{O}_{M}$ of slowly increasing smooth functions. In order to be able to apply functional analytic tools such as  De Wilde’s open mapping and closed graph theorems or the theory of the derived projective limit functor \cite{Wengenroth}, it is important to determine when such spaces are ultrabornological. Note that this is a non-trivial matter as the projective limit of a spectrum  of ultrabornological spaces is not necessarily again ultrabornological. The problem of characterizing when $(PLB)$-spaces are  ultrabornological has been extensively studied, both from an abstract point of view as for concrete function and (ultra)distribution spaces; see the survey article \cite{D-ClasPLSSpDist} and the references therein.

In the last part of his doctoral thesis \cite[Chap.\ 2, Th\'eor\`eme 16, p.\ 131]{G-ProdTensTopEspNucl} Grothendieck proved that the space  $\mathcal{O}_M$ is ultrabornological. He showed that $\mathcal{O}_M$ is isomorphic to a complemented subspace of $s \, \widehat{\otimes} \, s^{\prime}$ and verified directly that the latter space is ultrabornological. Later on Valdivia \cite{V-RepOM} showed that in fact $\mathcal{O}_{M}$ is isomorphic to  $s \, \widehat{\otimes} \, s^{\prime}$. A different proof of the fact that  $\mathcal{O}_{M}$ is ultrabornological was given by Larcher and  Wengenroth using homological methods \cite{L-W}.

In this article we study weighted $(PLB)$-spaces of ultradifferentiable functions.  Our spaces are defined as follows. 
Let $\omega : [0,\infty) \rightarrow [0,\infty)$ be a weight function (in the sense of Braun, Meise and Taylor \cite{B-M-T-UltraDiffFuncFAnal}) and set $\phi(x) = \omega(e^x)$. Denote by $\phi^{*}(y) = \sup_{x \geq 0} \{xy - \phi(x)\}$ the Young conjugate of $\phi$. Let $\V = \{v_\lambda ~  | ~  \lambda \in (0,\infty) \}$  be a family of continuous functions $v_\lambda: \R^d \rightarrow (0,\infty)$ such that  $1 \leq v_{\lambda} \leq v_{\mu}$ for all  $\mu \leq \lambda$. We call $\V$  a weight system. We then consider the weighted $(PLB)$-spaces of ultradifferentiable functions of Beurling and Roumieu type 
\[ \ZZ^{(\omega)}_{(\V)} := \varprojlim_{h \rightarrow 0^{+}} \varinjlim_{\lambda \rightarrow 0^+} \mathcal{Z}^{\omega,h}_{v_\lambda}, \qquad \ZZ^{\{\omega\}}_{\{\V\}} := \varprojlim_{\lambda \rightarrow \infty} \varinjlim_{h \rightarrow \infty} \mathcal{Z}^{\omega,h}_{v_\lambda}, \]
where $\mathcal{Z}^{\omega,h}_{v_\lambda}$ denotes the Banach space consisting of all $\varphi \in C^{\infty}(\R^{d})$ such that
	\[ \norm{\varphi}_{\mathcal{Z}^{\omega,h}_{v_\lambda}} := \sup_{\alpha \in \N^d} \sup_{x \in \R^d}\frac{|\varphi^{(\alpha)}(x)|}{v_\lambda(x)}\exp \left(-\frac{1}{h}\phi^*(h|\alpha|)\right) < \infty . \] 
We use $\ZZ^{[\omega]}_{[\V]}$ as a common notation for $\ZZ^{(\omega)}_{(\V)}$ and $\ZZ^{\{\omega\}}_{\{\V\}}$. The  first main goal of this article is to characterize  when $\ZZ^{[\omega]}_{[\V]}$ is ultrabornological through conditions on $\V$. These conditions will be closely related to  the linear topological invariants $(DN)$ and $(\overline{\overline{\Omega}})$  for Fr\'echet spaces \cite{V-FuncFSp}. Following Grothendieck, the key idea in our proof is to complement the space $\ZZ^{[\omega]}_{[\V]}$ into a suitable weighed $(PLB)$-space of continuous functions and, vice versa, to complement a suitable  weighted $(PLB)$-space of sequences into $\ZZ^{[\omega]}_{[\V]}$. Hereafter, we shall obtain the desired characterization by applying  results from \cite{A-B-B-ProjLimWeighLBSpContFunc} concerning the ultrabornologicity of such $(PLB)$-spaces. To achieve the first step, we use tools from time-frequency analysis \cite{G-FoundationsTimeFreqAnal}, specifically,  the short-time Fourier transform and Gabor frames. Such techniques have recently proved to be useful in the study of (generalized) function spaces; see e.g.\  \cite{B-O-CharSConvMulSpSTFT, D-N-V-UltraVanishInf, D-V-TopPropConvSTFT, KPSV}. 

Schwartz \cite{Schwartz} showed that $\mathcal{O}_M$ is equal to the multiplier space of the space $\mathcal{S}$ of rapidly decreasing smooth functions, i.e.,
$$
\mathcal{O}_M = \{ f \in \mathcal{S}' ~  | ~  \varphi \cdot f \in \mathcal{S} \mbox { for all $\varphi \in \mathcal{S}$} \}.
$$
Moreover, the natural $(PLB)$-space topology of $\mathcal{O}_M$ coincides with the topology induced by the embedding
	\[ \mathcal{O}_{M} \rightarrow L_{b}(\mathcal{S}, \mathcal{S}), \, f \mapsto ( \varphi \mapsto \varphi \cdot f) . \]	
The second main goal of this article is to obtain a similar result for a wide class of  Gelfand-Shilov spaces \cite{D-N-V-NuclGSKerThm}. Given a weight function $\omega$ and a weight system $\V$, we define the Gelfand-Shilov spaces of Beurling and Roumieu type as
\[ \mathcal{S}^{(\omega)}_{(\V)} := \varprojlim_{h \rightarrow 0^{+}} \mathcal{S}^{\omega,h}_{v_h}, \qquad \mathcal{S}^{\{\omega\}}_{\{\V\}} := \varinjlim_{h \rightarrow \infty} \mathcal{S}^{\omega,h}_{v_h}, \]
where $\mathcal{S}^{\omega,h}_{v_h}$ denotes the Banach space consisting of all $\varphi \in C^{\infty}(\R^{d})$ such that
	\[ \norm{\varphi}_{\mathcal{S}^{\omega,h}_{v_h}} = \sup_{\alpha \in \N^d} \sup_{x \in \R^d}|\varphi^{(\alpha)}(x)|v_h(x)\exp \left(-\frac{1}{h}\phi^*(h|\alpha|)\right) < \infty . \] 
We shall show that $\ZZ^{[\omega]}_{[\V]}$ is topologically equal to the multiplier space of $\mathcal{S}^{[\omega]}_{[\V]}$ . This problem has been previously studied for Fourier (ultra)hyperfunctions  \cite{KWL, Morimoto, Zharinov} and for general Gelfand-Shilov spaces of non-quasianalytic type \cite{D-P-V-MultConvTempUltraDist}. Our main improvement here is that we also consider the quasianalytic case and that, in contrast to the aforementioned works, we obtain topological and not merely algebraic identities.  Furthermore, by using the above results, we are able to determine when such multiplier spaces are ultrabornological. In particular, Theorem \ref{UB-3} below shows that the multiplier space of the space of the Fourier ultrahyperfunctions is ultrabornological, whereas the one of the space of Fourier hyperfunctions is not. We mention that analogous results for convolutor spaces of Gelfand-Shilov spaces have recently been obtained by Vindas and the first author \cite{D-V-WeightedIndUltra} (see also \cite{S-InclThMoyalMultAlgGenGSSp}).

The structure of this article is as follows. In the preliminary Sections \ref{sect-wf} and \ref{sect-ws} we define and study weight functions, weight sequences and weight systems. In Section \ref{sect:GS} we introduce Gelfand-Shilov spaces and discuss the short-time Fourier transform and Gabor frames in the context of these function spaces.  Our main results are stated and discussed in Section \ref{sect-main}. In the auxiliary Section \ref{sect-cont} we review some results from \cite{A-B-B-ProjLimWeighLBSpContFunc} about weighted $(PLB)$-spaces of continuous functions. Finally, the proofs of our main results are given in Section \ref{sect-proofs}. For this we study the short-time Fourier transform and Gabor frame expansions on various function spaces.

\section{Weight functions and weight sequences}\label{sect-wf}
A non-decreasing continuous function $\omega: [0,\infty) \rightarrow [0,\infty)$ is called a \emph{weight function} (in the sense of Braun, Meise and Taylor \cite{B-M-T-UltraDiffFuncFAnal}) if $\omega(0) = 0$ and $\omega$ satisfies the following properties:
	\begin{itemize}
		\item[$(\alpha)$] $\omega(2t) = O(\omega(t))$ as $t\to \infty$;
		\item[$(\gamma)$] $\log t = o(\omega(t))$ as $t\to \infty$;
		\item[$(\delta)$] $\phi: [0, \infty) \rightarrow [0, \infty)$, $\phi(x) = \omega(e^{x})$, is convex.
	\end{itemize}
We extend $\omega$ to $\R^d$ as the radial function $\omega(x) = \omega(|x|)$, $x \in \R^d$. Condition $(\alpha)$ implies that there is $C > 0$ such that \cite[Lemma 1]{B-M-T-UltraDiffFuncFAnal}
\begin{equation}
\label{alpha-extended}
\omega(x+y) \leq C(\omega(x)+ \omega(y) +1), \qquad x,y \in \R^d.
\end{equation}
A weight function $\omega$ is called \emph{non-quasianalytic} if
$$
\displaystyle \int_0^\infty \frac{\omega(t)}{1+t^2} \dt < \infty.
$$
We refer to \cite{B-M-T-UltraDiffFuncFAnal} for more information on these conditions. 

The \emph{Young conjugate of $\phi$} is defined as
	\[ \phi^{*} : [0, \infty) \rightarrow [0, \infty) , ~  \phi^{*}(y) = \sup_{x \geq 0} \{xy - \phi(x)\}. \]
The function $\phi^*$ is convex and increasing, $(\phi^*)^* = \phi$  and the function $y \mapsto \phi^*(y)/y$ is increasing on $[0,\infty)$ and tends to infinity as $y \to \infty$. We shall often use the following lemma.
\begin{lemma}\cite[Lemma 2.6]{Heinrich}\label{M12}
Let $\omega$ be a weight function. Then,
\begin{itemize}
\item[$(i)$] For all $h,k,l > 0$ there are $m,C > 0$ such that
\begin{equation}
 \frac{1}{m}\phi^*(m(y+l)) + ky \leq \frac{1}{h}\phi^*(hy) + \log C, \qquad  y \geq 0.
\label{ineq-0}
\end{equation}
\item[$(ii)$] For all $m,k,l > 0$ there are $h,C > 0$ such that \eqref{ineq-0} holds.
\end{itemize}
\end{lemma}
A sequence $M = (M_p)_{p \in \N}$ of positive numbers is called a \emph{weight sequence} \cite{K-Ultradistributions1}  if $M_p^{1/p} \rightarrow \infty$ as $p \to \infty$ and $M$ is log-convex, i.e.,  $M^2_p \leq M_{p-1}M_{p+1}$ for all $p \in \Z_+$. We set $m_p = M_{p}/M_{p-1}$, $p \in \Z_+$.
We consider the following  conditions on a weight sequence $M$:
\begin{itemize}
\item[$(M.2)' $]  $\ds M_{p+1} \leq CH^{p+1} M_{p}$, $p \in  \N$, for some $C,H > 0$;
\item[$(M.2)\ $]  $\ds M_{p+q} \leq CH^{p+q} M_{p} M_{q}$, $p,q\in \N$, for some $C,H > 0$;
\item[$(M.2)^*$] $2m_p \leq m_{Np}$, $p \geq p_0$, for some $p_0,N \in \Z_+$.
\end{itemize}
Clearly, $(M.2)$ implies $(M.2)'$. A weight sequence $M$ is called \emph{non-quasianalytic} if
$$
\sum_{p = 1}^\infty \frac{1}{m_p} < \infty.
$$
Conditions $(M.2)'$ and $(M.2)$  are due to Komatsu \cite{K-Ultradistributions1}. Condition $(M.2)^*$ was introduced by Bonet et al.\ \cite{B-M-M-CompDifferentClassUltraDiffFunc} without a name; we use here the same notation as in \cite{D-V-WeightedIndUltra}. The most important examples of weight sequences satisfying $(M.2)$ and $(M.2)^*$ are the \emph{Gevrey sequences} $p!^{s}$, $s>0$. The sequence $p!^{s}$ is non-quasianalytic if and only if $s > 1$.

Given two weight sequences $M$ and $N$, the relation $M \subset N$  means that there are $C,H > 0$ such that
$M_p\leq CH^{p}N_{p}$ for all $p\in\N$. The stronger relation $M \prec N$ means that the latter inequality is valid for every $H > 0$ and suitable $C > 0$. 

The \emph{associated function} of a weight sequence $M$ is defined as
	\[ \omega_{M}(t) =  \sup_{p \in \N}  \log \frac{t^p M_{0}}{M_{p}}, \qquad t \geq 0. \]
Given another weight sequence $N$, it holds that  $N \subset M$ if and only if 
$$
\omega_M(t) \leq \omega_N(Ht) + \log C, \qquad t \geq 0,
$$
for some $C,H > 0$ \cite[Lemma 3.8]{K-Ultradistributions1}. Similarly, $N \prec M$ if and only if the latter inequality remains valid for every $H>0$ and suitable $C>0$ \cite[Lemma 3.10]{K-Ultradistributions1}.

The next result explains when the weight sequence case can be reduced to the  weight function case.
\begin{lemma}  \cite[Proposition 13 and its proof]{B-M-M-CompDifferentClassUltraDiffFunc}
		\label{p:M2*equivBMT}
		Let $M$ be a weight sequence satisfying $(M.2)$. Then, $\omega_{M}$ is a weight function if and only if $M$ satisfies $(M.2)^{*}$. In such a case, the following properties hold (with $\phi_M(x) = \omega_M(e^{x})$):
\begin{itemize}
\item[$(i)$] For all $h> 0$ there are $k,C > 0$ such that
\begin{equation}
\exp\left(\frac{1}{k}\phi^*_M(kp)\right) \leq Ch^pM_p, \qquad p \in \N.
\label{ineq-1}
\end{equation}
\item[$(ii)$] For all $h> 0$ there are $k,C > 0$ such that
\begin{equation}
k^pM_p\leq C\exp \left (\frac{1}{h}\phi^*_M(hp) \right) , \qquad p \in \N.
\label{ineq-2}
\end{equation}
\item[$(iii)$] For all $k> 0$ there are $h,C > 0$ such that \eqref{ineq-1} holds.
\item[$(iv)$] For all $k> 0$ there are $h,C > 0$ such that \eqref{ineq-2} holds.
\end{itemize}
\end{lemma}

\section{Weight systems}\label{sect-ws} 
Let $X$ be a topological space.  A family $\V = \{ v_{\lambda} ~ |~  \lambda \in (0,\infty) \}$ of continuous functions $v_\lambda: X \rightarrow (0,\infty)$ is called a \emph{weight system} \cite{D-N-V-NuclGSKerThm} if  $1 \leq v_{\lambda}(x) \leq v_{\mu}(x)$ for all $x \in X$ and  $\mu \leq \lambda$.
The following two conditions play a crucial role in this article. 
\begin{definition}
		A weight system $\V$ on $X$ is said to satisfy condition $(\condDN)$ if
			\begin{gather*} 
				\exists \lambda ~ \forall \mu \leq \lambda ~ \forall \theta \in (0, 1) ~ \exists \nu \leq \mu  ~ \exists C > 0 ~ \forall x \in X ~:~ 
				v_\mu(x) \leq C v^{\theta}_\lambda(x) v_\nu^{1 - \theta}(x).
			\end{gather*}
	\end{definition}
	
	\begin{definition}
		A weight system $\V$ on $X$ is said to satisfy condition $(\condooOmega)$ if
			\begin{gather*}
				\forall \lambda  ~ \exists \mu \geq \lambda ~ \forall \nu \geq \mu   ~ \forall \theta \in (0, 1) ~ \exists C > 0 ~ \forall x \in X ~: ~ 
				v_{\mu}(x) \leq C v^\theta_{\lambda}(x) v^{1 - \theta}_{\nu}(x).
			\end{gather*}
	\end{definition}

	\begin{remark}
		The previous conditions are inspired by and closely related to the linear topological invariants $(DN)$ and $(\overline{\overline{\Omega}})$  for Fr\'echet spaces \cite{V-FuncFSp}.
	\end{remark}
Next, we consider weight systems on $\R^d$. We write $\widetilde{f}(t) = f(-t)$ for reflection about the origin. Given a weight function system $\V$ on $\R^d$, we write $\widetilde{\V} = \{ \widetilde{v}_\lambda ~  | ~  \lambda \in (0,\infty)\}$. We consider the following conditions on a weight system $\V$ on $\R^d$:
	\begin{itemize}
		\item[$(\condwM)$] $\forall \lambda ~ \exists \mu \leq \lambda ~ \exists C > 0 ~ \forall x \in \R^{d} ~ :~   \sup_{|y| \leq 1} v_{\lambda}(x + y) \leq C v_{\mu}(x)$;
		\item[$\{\condwM\}$] $\forall \mu  ~ \exists \lambda \geq \mu ~ \exists C > 0 ~ \forall x \in \R^{d} ~ :~  \sup_{|y| \leq 1} v_{\lambda}(x + y) \leq C v_{\mu}(x)$;
		\item[$(\condM)$] $\forall \lambda ~ \exists \mu, \nu \leq \lambda ~ \exists C > 0 ~ \forall x, y \in \R^{d} ~ :~  v_\lambda(x+y)\leq Cv_{\mu}(x) v_{\nu}(y) $; 
		\item[$\{\condM\}$] $\forall  \mu, \nu  ~ \exists \lambda \geq  \mu, \nu ~ \exists C > 0 ~ \forall x, y \in \R^{d} ~ :~  v_\lambda(x+y)\leq Cv_{\mu}(x) v_{\nu}(y)$;
		\item[$(\condN)$] $\forall \lambda  ~ \exists \mu \leq \lambda ~ :~  v_{\lambda} / v_{\mu} \in L^{1}$;
		\item[$\{\condN\}$] $\forall \mu  ~ \exists \lambda \geq \mu ~ :~  v_{\lambda} / v_{\mu} \in L^{1}$;
		\item[$(\condSquare)$] $\forall \lambda, \mu ~ \exists \nu \leq \lambda, \mu ~ \exists C > 0 ~ \forall x \in \R^{d} : v_{\lambda}(x) v_{\mu}(x) \leq C v_{\nu}(x)$; 
		\item[$\{\condSquare\}$] $\forall \nu   ~ \exists \lambda, \mu \geq \nu ~ \exists C > 0 ~ \forall x \in \R^{d} :v_{\lambda}(x) v_{\mu}(x) \leq C v_{\nu}(x)$.  
	\end{itemize}
\begin{notation}\label{notation}
 We employ $[\condwM]$ as a common notation for $(\condwM)$ and $\{\condwM\}$. A similar convention will be used for other notations. In addition, we  often first state assertions for the  Beurling case  followed in parenthesis by the corresponding ones for the Roumieu case.
\end{notation}	
	
Clearly, $[\operatorname{M}]$ implies $[\condwM]$.  A simple induction argument shows that $[\condwM]$ yields that
\begin{gather}
 \label{joo}
\forall a > 0 ~  \forall \lambda ~ \exists \lambda' \leq \lambda ~ ( \forall a > 0 ~ \forall \lambda'  ~ \exists \lambda \geq \lambda' ) \\ \nonumber
~ \exists C > 0 ~ \forall x \in \R^{d} ~  : ~ \sup_{|y| \leq a}v_{\lambda}(x+y) \leq C v_{\lambda'}(x).
\end{gather}
By  using the above formula twice, we obtain that  $[\condwM]$ implies that
\begin{gather*}
\forall a > 0 ~  \forall \lambda ~ \exists \lambda' \leq \lambda ~ \forall \mu' ~\exists \mu \leq \mu' ~ ( \forall a > 0 ~ \forall \lambda'  ~ \exists \lambda \geq \lambda' ~ \forall \mu ~ \exists \mu' \geq \mu ) \\ \nonumber
 \exists C > 0 ~ \forall x \in \R^{d} ~  : ~ \sup_{|y| \leq a}\frac{v_{\lambda}(x+y)}{v_{\mu}(x+y)} \leq C \frac{v_{\lambda'}(x)}{v_{\mu'}(x)}.
\end{gather*}
Consequently, $[\condwM]$ and $[\condN]$ imply that
\begin{equation}
\forall \lambda  ~ \exists \mu \leq \lambda  ~ ( \forall \mu  ~ \exists \lambda \geq \mu ) ~  : ~  \lim_{|x| \to \infty}\frac{v_{\lambda}(x)}{v_{\mu}(x)} = 0 \
\label{C0-wfs}
\end{equation}
and	
\begin{equation}
\forall a > 0 ~ \forall \lambda  ~ \exists \mu \leq \lambda ~ ( \forall a > 0 ~  \forall \mu  ~ \exists \lambda \geq \mu) ~ :~   \sum_{k  \in a \Z^d} \frac{v_\lambda(k)}{v_\mu(k)} < \infty.
\label{discrete-l1}
\end{equation}
We refer to \cite{D-N-V-NuclGSKerThm} for more information on these conditions.

We end this section by discussing the above conditions for two classes of weight systems on $\R^d$. 
Given a weight function $\omega$, we define
$$
\V_\omega := \{ e^{\frac{1}{\lambda} \omega } ~  | ~  \lambda \in (0,\infty) \}.
$$
\begin{lemma}\label{cond-wf}
Let $\omega$ be a weight function. Then, 
\begin{itemize}
\item[$(i)$]  $\V_\omega$ satisfies $[\condM]$, $[\condN]$ and $[\condSquare]$.
\item[$(ii)$] $\V_\omega$ satisfies $(\condDN)$.
\item[$(iii)$] $\V_\omega$ does not satisfy $(\condooOmega)$.
\end{itemize} 
\end{lemma}
\begin{proof}
$(i)$ Condition $[\condM]$ is a consequence of \eqref{alpha-extended}, $[\condN]$ follows from $(\gamma)$ and  $[\condSquare]$ is clear.

$(ii)$ This is obvious. 

$(iii)$ Since $\omega(t) \to \infty$ as $t \to \infty$, $(\condooOmega)$ for $\V_\omega$ would imply that
\begin{gather*}
				\forall \lambda  ~ \exists \mu \geq \lambda ~ \forall \nu \geq \mu   ~ \forall \theta \in (0, 1) ~  : ~  \frac{1}{\mu} \leq  \frac{\theta}{\lambda} +  \frac{1-\theta}{\nu},
								\end{gather*}
which is false. 								
\end{proof}
Given a weight sequence $M$, we define
$$
\V_M := \{ e^{\omega_M\left(\frac{1}{\lambda} \, \cdot \,  \right)} ~  | ~  \lambda \in (0,\infty) \}.
$$
\begin{lemma}\label{cond-seq}
Let $M$ be a weight sequence. Then, 
\begin{itemize}
\item[$(i)$]  $\V_M$ satisfies $[\condM]$.
\item[$(ii)$]  $\V_M$ satisfies  $[\condN]$ if and only if $M$ satisfies $(M.2)'$.
\item[$(iii)$]  $\V_M$ satisfies  $[\condSquare]$ if and only if $M$ satisfies $(M.2)$.
\item[$(iv)$] $\V_M$  satisfies $(\condDN)$.  
\item[$(v)$] $\V_M$ satisfies $(\condooOmega)$ if and only if 
\begin{equation}
\exists C > 0 ~  \forall N \in \Z_+ ~   \exists p_0 \in \Z_+ ~  \forall p \geq p_0 ~  : ~   m_{Np} \leq Cm_p.
\label{Omega-seq}
\end{equation}
\end{itemize} 
\end{lemma}
\begin{proof}
$(i)$ Since $\omega_M$ is increasing, we have that for all $H > 0$ and $x,y \in \R^d$
$$
\omega_M(H(x+y)) \leq \omega_M(2H\max\{ |x|, |y|\}) \leq \omega_M(2Hx) + \omega_M(2Hy).
$$
This implies that $\V_M$ satisfies $[\condM]$.

$(ii)$ This is shown in  \cite[Lemma 3.3]{D-N-V-NuclGSKerThm}.

$(iii)$ This follows from \cite[Proposition 3.6]{K-Ultradistributions1}.

$(iv)$ For all $H > 0$ and $\theta \in (0,1)$ it holds that
\begin{align*}
\omega_M(Ht) 
&= \sup_{p\in \N} \left ( \theta \log \left(\frac{t^p M_0}{M_p} \right) + (1-\theta) \log \left ( \frac{(H^{1/(1-\theta)}t)^p M_0}{M_p} \right) \right) \\
&\leq \theta \omega_M(t) + (1-\theta)\omega_M(H^{1/(1-\theta)}t), 
\end{align*}
for all $t \geq 0$. This shows that $\V_M$  satisfies $(\condDN)$.

$(v)$ We denote by $m$ the counting  function of the sequence $(m_p)_{p \in \Z_+}$, i.e., 
$$
m(x) = \sum_{m_p \leq x} 1, \qquad x \geq 0. 
$$
Then, \cite[Equation (3.11)]{K-Ultradistributions1}
$$
\omega_M(t) = \int_{0}^{t} \frac{m(x)}{x} \dx, \qquad t \geq 0.
$$
Hence, $\V_M$ satisfies $(\condooOmega)$ if and only if 
\begin{equation}
\forall H > 0 ~ \exists  K < H ~  \forall L \leq K ~  : ~  \int_{L t}^{K t} \frac{m(x)}{x} \dx = o \left ( \int^{Ht}_{L t} \frac{m(x)}{x} \dx \right),
\label{omega-new}
\end{equation}
while \eqref{Omega-seq} holds if and only if
\begin{equation}
\exists C > 1 ~  : ~  m(x) = o  (m(Cx)).
\label{cond-m}
\end{equation}
We now show that \eqref{omega-new} and \eqref{cond-m} are equivalent. First assume that \eqref{omega-new} holds. Let $\varepsilon > 0$ be arbitrary. Condition \eqref{omega-new} with $H = 1$ and  $L = K/e$ implies that for $t$ large enough
$$
m (K t/e ) \leq \int^{K t}_{K t / e} \frac{m(x)}{x} \dx \leq \varepsilon \int^{t}_{K t / e} \frac{m(x)}{x} \dx \leq \varepsilon \log( e/K) m(t),
$$
whence \eqref{cond-m} holds (with $C = e/K$). Conversely, assume that \eqref{cond-m} holds. Let $H > 0$ be arbitrary and set $K = H/C$. Fix $L \leq K$. Let $\varepsilon > 0$ be arbitrary. Condition  \eqref{cond-m} implies that for all $t$ large enough
$$
\int_{L t}^{K t} \frac{m(x)}{x} \dx  \leq \varepsilon \int_{L t}^{K t} \frac{m(Cx)}{x} \dx = \varepsilon \int_{L C t}^{Ht} \frac{m(x)}{x} \dx \leq  \varepsilon \int_{L t}^{Ht} \frac{m(x)}{x} \dx,
$$
whence \eqref{omega-new} holds.
\end{proof}
\begin{example} \label{example} 
$(i)$ Consider the weight sequence $M = ((\log (p+e))^{sp})_{p \in \N}$, $s > 0$. Since $M$ is log-convex and satisfies $(M.2)$, we have that $0 < \sup_{p \in \N}m_p/M^{1/p}_p < \infty$. Hence, there is $C> 0$ such that $C^{-1}(\log (p+e))^{s} \leq m_p \leq C(\log (p+e))^{s}$ for all $p \in \N$. This implies that $M$ satisfies  \eqref{Omega-seq}.

$(ii)$ A simple induction argument shows that $(M.2)^*$  yields that
$$
\forall C > 0 ~  \exists N \in \Z_+ ~   \exists p_0 \in \Z_+ ~  \forall p \geq p_0 ~  : ~   Cm_{p} \leq m_{Np}.
$$
Therefore, any weight sequence satisfying $(M.2)^*$ does not satisfy \eqref{Omega-seq}. In particular, the Gevrey sequence $p!^s$, $s>0$, does not satisfy \eqref{Omega-seq}.
\end{example}
\section{Gelfand-Shilov spaces and time-frequency analysis}\label{sect:GS}
Let $\omega$ be a weight function. For $h > 0$  and a continuous function $v: \R^d \to (0,\infty)$ we define
$\mathcal{S}^{\omega,h}_{v}$ as the Banach space consisting of all $\varphi \in C^{\infty}(\R^{d})$ such that
	\[ \norm{\varphi}_{\mathcal{S}^{\omega,h}_{v} } := \sup_{\alpha \in \N^{d}} \sup_{ x \in \R^{d}} |\varphi^{(\alpha)}(x)|v(x)\exp \left(-\frac{1}{h}\phi^*(h|\alpha|)\right) < \infty . \] 
Let $\V$ be a weight system (on $\R^d$). We define the \emph{Gelfand-Shilov spaces of Beurling and Roumieu type} as
	\[ \mathcal{S}^{(\omega)}_{(\V)}  := \varprojlim_{h \rightarrow 0^{+}} \mathcal{S}^{\omega,h}_{v_h} , \qquad \mathcal{S}^{\{\omega\}}_{\{\V\}} := \varinjlim_{h \rightarrow \infty} \mathcal{S}^{\omega,h}_{v_h} . \]
Then, $\mathcal{S}^{(\omega)}_{(\V)}$ is a Fr\'{e}chet space and $\mathcal{S}^{\{\omega\}}_{\{\V\}}$ is an $(LB)$-space.  Following Notation \ref{notation}, we employ $\mathcal{S}^{[\omega]}_{[\V]}$ as a common notation for $\mathcal{S}^{(\omega)}_{(\V)}$ and $\mathcal{S}^{\{\omega\}}_{\{\V\}}$.  If $\V$ satisfies $[\condwM]$, then $\mathcal{S}^{[\omega]}_{[\V]}$ is translation-invariant. If $\V$ satisfies $[\condN]$, then $\mathcal{S}^{[\omega]}_{[\V]} \subset L^1 \cap L^\infty \subset L^p$ for all $p \in [1,\infty]$.  We refer to \cite{D-N-V-NuclGSKerThm} for more information on $\mathcal{S}^{[\omega]}_{[\V]}$. Given another weight function $\eta$, we write $\mathcal{S}^{[\omega]}_{[\eta]} := \mathcal{S}^{[\omega]}_{[\V_\eta]}$.

Let $M$ and $A$ be two weight sequences.  For $h > 0$ we define $\mathcal{S}^{M,h}_{A,h}$ as the Banach space consisting of all $\varphi \in C^{\infty}(\R^{d})$ such that
	\[ \norm{\varphi}_{\mathcal{S}^{M,h}_{A,h} } := \sup_{\alpha,\beta \in \N^{d}} \sup_{ x \in \R^{d}} \frac{|x^\beta\varphi^{(\alpha)}(x)|}{h^{|\alpha|+ |\beta|}M_{|\alpha|} A_{|\beta|}}  < \infty . \]
We define
	\[ \mathcal{S}^{(M)}_{(A)} := \varprojlim_{h \rightarrow 0^{+}} \mathcal{S}^{M,h}_{A,h} , \qquad \mathcal{S}^{\{M\}}_{\{A\}} := \varinjlim_{h \rightarrow \infty} \mathcal{S}^{M,h}_{A,h} . \]
Then, $\mathcal{S}^{(M)}_{(A)}$ is a Fr\'{e}chet space and $\mathcal{S}^{\{M\}}_{\{A\}}$ is an $(LB)$-space. 

\begin{lemma}\label{GS-red}
Let $M$ and $A$ be two weight sequences. Suppose that $M$ satisfies $(M.2)$ and $(M.2)^*$. Then,
$\mathcal{S}^{[M]}_{[A]} = \mathcal{S}^{[\omega_M]}_{[\V_A]}$ as locally convex spaces. 
\end{lemma}
\begin{proof}
This follows from Lemma \ref{p:M2*equivBMT} and the fact that for all $h > 0$
$$
\omega_A \left( \frac{1}{\sqrt{d}h} |x| \right) \leq \sup_{\beta \in \N^d} \log \frac{|x^\beta|}{h^{|\beta|}A_{|\beta|}} \leq \omega_A \left( \frac{1}{h} |x| \right), \qquad x \in \R^d.
$$
\end{proof}
Let $r,s > 0$. We write
 $$
 \Sigma^r_s := \mathcal{S}^{(p!^r)}_{(p!^s)} = \mathcal{S}^{(t^{1/r})}_{(t^{1/s})}  , \qquad \mathcal{S}^r_s := \mathcal{S}^{\{p!^r\}}_{\{p!^s\}} = \mathcal{S}^{\{t^{1/r}\}}_{\{ t^{1/s}\}}, \qquad
 $$
 for the classical Gelfand-Shilov spaces \cite{G-S-GenFunc2}. In particular, $\Sigma^1_1$ is the test function space of the Fourier ultrahyperfunctions \cite{Zharinov} and $\mathcal{S}^1_1$ is the test function space of the Fourier hyperfunctions \cite{Kawai}.
 
 \begin{remark}\label{non-triviality}
The space $\Sigma^r_s$ ($\mathcal{S}^r_s$) is non-trivial if and only if $r+s > 1$ ($r+s \geq 1$) (cf.\ \cite[Section 8]{G-S-GenFunc2}).
Consequently, given a weight function $\omega$ and a weight system $\V$, we have that  $\mathcal{S}^{[\omega]}_{[\V]} \neq \{0\}$ if  $\omega(t) = O(t^{1/r})$ and 
$$
\forall \lambda ~   \exists h  ~  ( \forall h ~  \exists \lambda) ~  : ~  v_\lambda(x) = O(e^{h|x|^{1/s}})
$$
for some $r+s > 1$ ($r+s \geq 1$), as these conditions imply that $\Sigma^r_s \subseteq \mathcal{S}^{(\omega)}_{(\V)}$ ($\mathcal{S}^r_s \subseteq \mathcal{S}^{\{\omega\}}_{\{\V\}}$). In particular, if $\eta$ is another weight function, $\mathcal{S}^{[\omega]}_{[\eta]} \neq \{0\}$ if  $\omega(t) = O(t^{1/r})$  and $\eta(t) = O(t^{1/s})$ for some $r+s > 1$ ($r+s \geq 1$). Similarly, given two weight sequences $M$ and $A$, $\mathcal{S}^{[M]}_{[A]} \neq \{ 0\}$ if $p!^r \subset M$ and $p!^s \subset A$  for some $r+s > 1$ ($r+s \geq 1$).  In \cite[Proposition 2.7, Proposition 4.3 and Theorem 5.9]{D-V2018} Vindas and the first author showed that  $\mathcal{S}^{[p!]}_{[A]} \neq \{0\}$ if and only if $(\log (p+e))^p \prec A$ ($(\log (p+e))^p \subset A$). In general, the characterization of the non-triviality of the spaces $\mathcal{S}^{[\omega]}_{[\eta]}$ and $\mathcal{S}^{[M]}_{[A]}$ seems to be an open problem. 
\end{remark}
Next, we introduce some tools from time-frequency analysis; see the book \cite{G-FoundationsTimeFreqAnal} for more information. The translation and modulation operators are denoted by $T_{x} f(t) = f(t - x)$ and $M_{\xi} f(t) = e^{2 \pi i \xi \cdot t} f(t)$, for $x, \xi \in \R^{d}$. 
The \emph{short-time Fourier transform (STFT)} of $f \in L^{2}(\R^{d})$ with respect to a window  $\psi \in L^{2}(\R^{d})$ is defined as 
	\[ V_{\psi} f(x, \xi) = (f, M_{\xi} T_{x} \psi)_{L^{2}} = \int_{\R^{d}} f(t) \overline{\psi(t - x)} e^{- 2 \pi i \xi \cdot t} dt , \qquad (x, \xi) \in \R^{2d} .  \]
We have that $\|V_{\psi} f\|_{L^2} = \|\psi \|_{L^2} \|f \|_{L^2}$. In particular, $V_\psi : L^{2}(\R^d) \rightarrow L^2(\R^{2d})$ is continuous. The adjoint of $V_{\psi}$ is given by the weak integral 
	\[ V_{\psi}^{*} F = \int \int_{\R^{2d}} F(x, \xi) M_{\xi} T_{x} \psi dx d\xi , \qquad F \in L^{2}(\R^{2d}) . \]
If  $\gamma \in L^{2}(\R^{d})$ is such that $(\gamma, \psi)_{L^{2}} \neq 0$, then
	\begin{equation}
		\label{eq:reconstructSTFTL2} 
		\frac{1}{(\gamma, \psi)_{L^{2}}} V^{*}_{\gamma} \circ V_{\psi} = \id_{L^{2}(\R^{d})} . 
	\end{equation}
The above reconstruction formula is the basis for the proof of part $(a)$ of Theorem \ref{t:ZSpTop} below.
	
Next, we consider Gabor frames.  Given a window  $\psi \in L^{2}(\R^{d})$ and lattice parameters $a, b > 0$, the set of time-frequency shifts
	\[ \mathcal{G}(\psi, a, b) := \{ M_{n} T_{k} \psi : (k,n) \in a\Z^{d} \times b \Z^d \} \]
is called a \emph{Gabor frame} for $L^{2}(\R^{d})$ if there exist $A, B > 0$ such that
	\[ A \norm{f}^{2}_{L^{2}} \leq \sum_{ (k,n) \in a\Z^{d} \times b \Z^d } \left| V_{\psi} f(k, n) \right|^{2} \leq B \norm{f}^{2}_{L^{2}} , \qquad f  \in L^{2}(\R^{d}). \]
We define the Wiener space $W$ as the space consisting of all $\psi \in L^\infty(\R^d)$ such that
$$
\sum_{n \in \Z^d} \| T_n \psi \|_{L^\infty([0,1]^d)} < \infty.
$$
Given a weight function $\omega$ and a weight system $\V$ satisfying $[\condwM]$ and $[\condN]$, we have that $\mathcal{S}^{[\omega]}_{[\V]} \subset W$. This follows from $[\condwM]$ and the fact that for some $\mu > 0$ (for all $\mu > 0$) we have that
$$
\sum_{n \in \Z^d} \frac{1}{v_\mu(n)} < \infty
$$
(the latter is a consequence of \eqref{discrete-l1} and the fact that $v_\lambda \geq 1$ for all $\lambda > 0$). Let $\psi \in W$. Then, the \emph{analysis operator}
$$
C_\psi = C^{a,b}_\psi:  L^2(\R^d) \rightarrow l^2( a\Z^{d} \times b \Z^d), ~  f \mapsto (V_\psi f(k,n))_{ (k,n) \in a\Z^{d} \times b \Z^d},
$$
and the \emph{synthesis operator}
$$
D_\psi = D^{a,b}_\psi: l^2(a\Z^{d} \times b \Z^d) \rightarrow L^2(\R^d), ~  (c_{k,n})_{ (k,n) \in a\Z^{d} \times b \Z^d} \mapsto \sum_{ (k,n) \in a\Z^{d} \times b \Z^d} c_{k,n} M_{n} T_{k} \psi
$$
are continuous \cite[Proposition 6.2.2 and Corollary 6.2.3]{G-FoundationsTimeFreqAnal}. Given another window $\gamma \in W$, we define
$$
S_{\psi,\gamma} := D_\gamma \circ C_\psi: L^2(\R^d) \rightarrow L^2(\R^d).
$$
We call  \emph{$(\psi,\gamma)$ a pair of dual windows  (on $a\Z^d \times b\Z^d$)} if $S_{\psi,\gamma} = \operatorname{id}_{L^2(\R^d)}$. In such a case, also $S_{\gamma, \psi} = \operatorname{id}_{L^2(\R^d)}$ and both $\mathcal{G}(\psi, a, b)$ and $\mathcal{G}(\gamma, a, b)$ are Gabor frames.  
Pairs of dual windows  are characterized by the Wexler-Raz biorthogonality relations:
\begin{theorem}\cite[Theorem 7.3.1 and the subsequent remark]{G-FoundationsTimeFreqAnal}
		\label{l:WZBiOrthRel}
		Let $\psi, \gamma \in W$ and let $a,b > 0$. Then,  $(\psi,\gamma)$  is a pair of dual windows on $a\Z^d \times b\Z^d$ if and only if
		$$
		( M_{n} T_{k}\psi, M_{n'} T_{k'} \gamma)_{L^{2}} = (ab)^{d} \delta_{k, k'} \delta_{n,n' }, \qquad (k, n), (k,' n') \in \frac{1}{a}\Z^d \times  \frac{1}{b}\Z^d,
$$
or, equivalently,
\begin{equation}
\frac{1}{(ab)^{d}}C^{ 1/a, 1/b}_\psi \circ D^{1/a,1/b}_\gamma =  \operatorname{id}_{l^2\left( \frac{1}{a}\Z^d \times  \frac{1}{b}\Z^d \right)}.
		\label{WR}
\end{equation}
\end{theorem}		
The proof of part $(b)$ of Theorem  \ref{t:ZSpTop} below is based on the formula \eqref{WR}. For it to be applicable in our context we need that, given a weight function $\omega$ and a weight system $\V$, $\psi \in \mathcal{S}^{[\omega]}_{[\V]}$ and $\gamma \in \mathcal{S}^{[\omega]}_{[\widetilde \V]}$. Hence, we introduce the following general notion:

	\begin{definition}
		Let $\omega$ be  a weight function and let $\V$ be a  weight system. The space  $\mathcal{S}^{[\omega]}_{[\V]}$ is called \emph{Gabor accessible} if there exist $\psi \in \mathcal{S}^{[\omega]}_{[\V]}$, $\gamma \in \mathcal{S}^{[\omega]}_{[\widetilde \V]}$ and $a, b > 0$ such that $(\psi,\gamma)$  is a pair of dual windows on $a\Z^d \times b\Z^d$.
	\end{definition}
The regularity and decay properties of pairs of dual windows is a well-studied topic in time-frequency analysis; see \cite[Chapter 13]{G-FoundationsTimeFreqAnal} and the references therein. We now use such results to give growth conditions on  $\omega$ and $\V$ which ensure that  $\mathcal{S}^{[\omega]}_{[\V]}$ is Gabor accessible.
\begin{proposition}\label{GA-1} Let $\omega$ be  a weight function and let $\V$ be a  weight system. Then, $\mathcal{S}^{[\omega]}_{[\V]}$ is Gabor accessible if one of the following two conditions is satisfied:
\begin{itemize}
\item[$(i)$] $\omega$ is non-quasianalytic.
\item[$(ii)$] $\omega(t) = o(t^2)$ and $\forall \lambda ~ \forall h ~: ~ v_\lambda(x) = O(e^{h|x|^2})$   ($\omega(t) = O(t^2)$ and $\forall h ~\exists \lambda ~: ~ v_\lambda(x) = O(e^{h|x|^2})$).
\end{itemize}
\end{proposition}
\begin{proof}
Theorem \ref{l:WZBiOrthRel} implies that  if $(\psi, \gamma) \subset W(\R)$ is a pair of dual windows on $a\Z \times b \Z$, $a,b > 0$, then $(\psi \otimes  \cdots \otimes \psi, \gamma \otimes \cdots \otimes \gamma)\subset W(\R^d)$ is a pair of dual windows on $a\Z^d \times b \Z^d$. Now assume that $(i)$ holds. Then, there exists a  function $\psi: \R \rightarrow \R$ with $\operatorname{supp} \psi \subseteq [0,2]$ such that
$$
 \sup_{p \in \N} \sup_{ x \in [0,2]} |\psi^{(p)}(x)|\exp \left(-\frac{1}{h}\phi^*(hp)\right) < \infty 
$$
for all $h > 0$ and 
$$
\sum_{k \in \Z} T_k \psi = 1.
$$
Fix $0 < b \leq 1/3$. Define
$$
\gamma(x) = b\psi(x) + 2b \psi(x+1), \qquad x \in \R.
$$
In \cite[Theorem 2.2]{C-PairsDualGaborFrameGenCompSupp} it is shown that $(\psi, \gamma)$ is a pair of dual windows on $\Z \times b \Z$. By the remark at the beginning of the proof, we obtain that $(\psi \otimes  \cdots \otimes \psi, \gamma \otimes \cdots \otimes \gamma)\ \subset   \mathcal{S}^{[\omega]}_{[\V]}(\R^d) \cap \mathcal{S}^{[\omega]}_{[\widetilde \V]}(\R^d)$ is a pair of dual windows on $\Z^d \times b \Z^d$. Next, assume that $(ii)$ holds. This condition implies that $\mathcal{S}^{1/2}_{1/2}(\R^d) \subseteq \mathcal{S}^{[\omega]}_{[\V]}(\R^d) \cap \mathcal{S}^{[\omega]}_{[\widetilde \V]}(\R^d)$. Hence, it suffices to show that $\mathcal{S}^{1/2}_{1/2}(\R^d)$ is Gabor accessible. Moreover, by the remark at the beginning of the proof, it is enough to consider the case $d = 1$. Set $\psi(x) = e^{-\pi x^2}$, $x \in \R$. Then,  $\psi \in  \mathcal{S}^{1/2}_{1/2}(\R)$ and Janssen \cite[Proposition B and its proof]{Janssen} showed that for all $a,b > 0$ with $ab <1$ there exists $\gamma \in  \mathcal{S}^{1/2}_{1/2}(\R)$ such that $(\psi, \gamma)$ is a pair of dual windows on $a\Z \times b\Z$
(see also \cite[p.\ 273]{B-J-GaborUnimodWindDecay}).
\end{proof}
Next, we discuss the Gabor accessibility of the spaces $\mathcal{S}^{[\omega]}_{[\eta]}$ and $\mathcal{S}^{[M]}_{[A]}$.
\begin{proposition}\label{GA-2}
Let $\omega$ and $\eta$ be two weight functions. Then, $\mathcal{S}^{[\omega]}_{[\eta]}$ is Gabor accessible if one of the following two conditions is satisfied:
\begin{itemize}
\item[$(i)$] $\omega$ or $\eta$ is non-quasianalytic.
\item[$(ii)$] $\omega(t) = o(t^2)$ and $\eta(t) = o(t^2)$ ($\omega(t) = O(t^2)$ and $\eta(t) = O(t^2)$).
\end{itemize}
\end{proposition}
\begin{proof} If $\omega$ is non-quasianalytic or $(ii)$ is satisfied, the result is a direct consequence of Proposition \ref{GA-1}. Now assume that $\eta$ is non-quasianalytic. Since the Fourier transform is an isomorphism from $\mathcal{S}^{[\omega]}_{[\eta]}$ onto $\mathcal{S}_{[\omega]}^{[\eta]}$ and $(\psi, \gamma) \subseteq \mathcal{S}$ is a pair of dual windows on $a\Z^d \times b\Z^d$, $a,b > 0$ if and only if $(\widehat{\psi}, \widehat{\gamma})$ is a pair of dual windows on $b\Z^d \times a\Z^d$ (as follows from Theorem \ref{l:WZBiOrthRel} and Plancherel's theorem), the space $\mathcal{S}^{[\omega]}_{[\eta]}$ is Gabor accessible because $\mathcal{S}^{[\eta]}_{[\omega]}$ is so.
\end{proof}	

\begin{proposition}\label{GA-3}
Let $M$ and $A$ be two weight functions satisfying $(M.2)$. Then, $\mathcal{S}^{[M]}_{[A]}$ is Gabor accessible if one of the following two conditions is satisfied:
\begin{itemize}
\item[$(i)$] $M$ or $A$ is non-quasianalytic.
\item[$(ii)$] $p!^{1/2} \prec M$ and $p!^{1/2} \prec A$ ($p!^{1/2} \subset M$ and $p!^{1/2} \subset A$ ).
\end{itemize}
\end{proposition}
\begin{proof}
 If $M$ is non-quasianalytic or $(ii)$ is satisfied, the result can be shown in the same way as Proposition \ref{GA-1}. Now assume that $A$ is non-quasianalytic. Since the Fourier transform is an isomorphism from $\mathcal{S}_{[A]}^{[M]}$ onto $\mathcal{S}_{[M]}^{[A]}$,  the result can be shown by using the same argument as in the proof of Proposition \ref{GA-2}.
\end{proof}	
\begin{proposition}\label{GA-4}
Let $r,s > 0$. Then, $\Sigma^r_s$ ($\mathcal{S}^r_s$) is Gabor accessible if  $\max(r, s) > 1$ or $\min(r, s) > 1/2$ ($\max(r, s) > 1$ or $\min(r, s) \geq 1/2$).
\end{proposition}
\begin{proof}
This follows from Proposition \ref{GA-3}.
\end{proof}
	
Finally, we would like to point out the following open problem.
\begin{problem}
Let $r,s > 0$. Is every non-trivial space $\mathcal{S}^r_s$ Gabor accessible? This would imply that every non-trivial space  $\Sigma^r_s$ is Gabor accessible.
If not, characterize the Gabor accessibility of  the spaces $\Sigma^r_s$  and $\mathcal{S}^r_s$ in terms of  $r$ and $s$.
\end{problem}

\section{Statement of the main results}\label{sect-main}
Let $\omega$ be a weight function and let $\V$ be  a weight system (on $\R^d$). We define 
	\[ \ZZ^{(\omega)}_{(\V)} := \varprojlim_{h \rightarrow 0^{+}} \varinjlim_{\lambda \rightarrow 0^+} \mathcal{S}^{\omega,h}_{1/v_\lambda}, \qquad \ZZ^{\{\omega\}}_{(\V)} := \varprojlim_{\lambda \rightarrow \infty} \varinjlim_{h \rightarrow \infty} \mathcal{S}^{\omega,h}_{1/v_\lambda}. \]
Then, $\ZZ^{[\omega]}_{[\V]}$ is a $(PLB)$-space. The first main result of this article may now be formulated as follows.
\begin{theorem}
		\label{t:ZSpTop}
		Let $\omega$ be a weight function and let $\V$ be a weight system satisfying $[\condM]$ and $[\condN]$. Consider the following  statements:
			\begin{itemize}
				\item[$(i)$] $\V$ satisfies $(\condDN)$ ($(\condooOmega)$).
				\item[$(ii)$] $\ZZ^{[\omega]}_{[\V]}$  is ultrabornological.
				\item[$(iii)$] $\ZZ^{[\omega]}_{[\V]}$ is barrelled. 
			\end{itemize}
		Then, 
			\begin{itemize}
				\item[$(a)$]  If $\mathcal{S}^{[\omega]}_{[\V]} \neq \{0\}$, then $(i) \Rightarrow (ii) \Rightarrow (iii)$.
				\item[$(b)$] If $\mathcal{S}^{[\omega]}_{[\V]}$ is Gabor accessible, then also $(iii) \Rightarrow (i)$.
			\end{itemize}
	\end{theorem}
The assumption that $\mathcal{S}^{[\omega]}_{[\V]}$ is non-trivial and Gabor accessible in part $(a)$ and part $(b)$ of Theorem \ref{t:ZSpTop}, respectively, should be interpreted as implicit growth conditions on $\omega$ and $\V$ under which these results are valid. We refer to Remark \ref{non-triviality} and Proposition \ref{GA-1} for explicit conditions on $\omega$ and $\V$ which ensure that $\mathcal{S}^{[\omega]}_{[\V]}$ is non-trivial and Gabor accessible, respectively.

Next, we discuss our results about multiplier spaces. We need some preparation. Given a weight function $\omega$ and  a weight system $\V$,  we denote by $\mathcal{S}'^{[\omega]}_{[\V]}$ the strong dual of $\mathcal{S}^{[\omega]}_{[\V]}$.  We write $C_{[\V]}$ for the space consisting of all $f \in C(\R^d)$ such that $\sup_{x \in \R^d} |f(x)|/v_\lambda(x) < \infty$ for some $\lambda > 0$ (for all $\lambda > 0$). Note that $\mathcal{Z}^{[\omega]}_{[\V]} \subset C_{[\V]}$.

\begin{lemma} \label{new}
 Let $\omega$ be a weight function and let $\V$ be a weight system satisfying $[\condwM]$ and $[\condN]$. Suppose that $\mathcal{S}^{[\omega]}_{[\V]} \neq \{ 0\}$. The mapping 
\begin{equation}
 C_{[\V]} \rightarrow \mathcal{S}'^{[\omega]}_{[\V]}, \, f \mapsto \left ( \varphi \mapsto  \int_{\R^d} f(x) \varphi(x) \dx \right)
\label{mapping}
\end{equation}
is well-defined and injective. Consequently, we may view $C_{[\V]}$  as a vector subspace of $\mathcal{S}'^{[\omega]}_{[\V]}$.
\end{lemma}
\begin{proof}
Condition $[\condN]$ implies that, for each $f \in C_{[\V]}$,  
$$
\langle f, \varphi \rangle = \int_{\R^d} f(x) \varphi(x) \dx, \qquad \varphi \in \mathcal{S}^{[\omega]}_{[\V]},
$$
is a well-defined  continuous linear functional on $\mathcal{S}^{[\omega]}_{[\V]}$. We now show that the mapping \eqref{mapping} is injective. Let $f \in C_{[\V]}$ be such that $\langle f, \varphi \rangle = 0$ for all $\varphi \in \mathcal{S}^{[\omega]}_{[\V]}$. Since the space  $\mathcal{S}^{[\omega]}_{[\V]}$ is translation-invariant and non-trivial, there exists $\varphi \in  \mathcal{S}^{[\omega]}_{[\V]}$ with $\varphi(0) = 1$. Choose $\chi \in \mathcal{D}(\R^d)$ with $\chi(0) = 1$. Set $\psi = \widehat{\chi}$ and note that $\int_{\R^d} \psi(x)\dx = 1$. We write $\psi_n(x) = n^d \psi(nx)$ for $n \in \N$. Lemma \ref{M12} implies that $\varphi \widetilde{\psi}_n \in  \mathcal{S}^{[\omega]}_{[\V]}$ for all $n \in \N$. Fix an arbitrary $x \in \R^{d}$. Since $fT_x \varphi \in C(\R^d) \cap L^\infty$, we have that
$$
f(x) = f(x)\varphi(0) = (fT_x\varphi)(x) = \lim_{n \to \infty} (fT_x\varphi) \ast \psi_n (x) = \lim_{n \to \infty}  \langle f, T_x(\varphi \widetilde{\psi}_n) \rangle = 0.
$$
\end{proof}
Let $\omega$ be a weight function and let $\V$ be a weight system satisfying the assumptions of Lemma \ref{new}. The space $\mathcal{S}^{[\omega]}_{[\V]}$ is an algebra under pointwise multiplication and  the mapping
$
\mathcal{S}^{[\omega]}_{[\V]} \times \mathcal{S}^{[\omega]}_{[\V]} \rightarrow \mathcal{S}^{[\omega]}_{[\V]}, ~  (\varphi, \psi) \mapsto \varphi \cdot \psi
$
is separately continuous.  For $f \in \mathcal{S}'^{[\omega]}_{[\V]}$ and $\varphi \in \mathcal{S}^{[\omega]}_{[\V]}$  we define $\varphi \cdot f \in  \mathcal{S}'^{[\omega]}_{[\V]}$ via transposition, i.e.,
$\langle \varphi \cdot f, \psi \rangle := \langle f, \varphi \cdot \psi \rangle$ for $\psi \in \mathcal{S}^{[\omega]}_{[\V]}$.
Then, the mapping 
$
\mathcal{S}^{[\omega]}_{[\V]}  \times \mathcal{S}'^{[\omega]}_{[\V]} \rightarrow \mathcal{S}'^{[\omega]}_{[\V]}, ~  (\varphi, f) \mapsto \varphi \cdot f
$
is separately continuous. We define the \emph{multiplier space of $\mathcal{S}^{[\omega]}_{[\V]}$} as
$$
\mathcal{O}_{M}(\mathcal{S}^{[\omega]}_{[\V]}) := \{ f \in  \mathcal{S}'^{[\omega]}_{[\V]} ~  | ~   \varphi \cdot f \in \mathcal{S}^{[\omega]}_{[\V]} \mbox{ for all } \varphi \in \mathcal{S}^{[\omega]}_{[\V]} \}.
$$
Fix $f \in \mathcal{O}_{M}(\mathcal{S}^{[\omega]}_{[\V]})$. The closed graph theorem of De Wilde and the continuity of the mapping $
\mathcal{S}^{[\omega]}_{[\V]} \rightarrow \mathcal{S}'^{[\omega]}_{[\V]}, ~  \varphi \mapsto \varphi \cdot f
$ imply that the mapping
$
\mathcal{S}^{[\omega]}_{[\V]} \rightarrow \mathcal{S}^{[\omega]}_{[\V]}, ~  \varphi \mapsto \varphi \cdot f
$
is continuous. We endow $\mathcal{O}_{M}(\mathcal{S}^{[\omega]}_{[\V]})$ with the topology induced by the embedding
	\[ \mathcal{O}_{M}(\mathcal{S}^{[\omega]}_{[\V]}) \rightarrow L_{b}(\mathcal{S}^{[\omega]}_{[\V]}, \mathcal{S}^{[\omega]}_{[\V]}), \, f \mapsto ( \varphi \mapsto \varphi \cdot f) . \]	
We then have:
\begin{theorem}\label{main-2}
		Let $\omega$ be a weight function and let $\V$ be a weight system satisfying $[\condM]$, $[\condN]$ and $[\condSquare]$. Suppose that $\mathcal{S}^{[\omega]}_{[\V]} \neq \{0\}$. Then,  $\mathcal{O}_{M}(\mathcal{S}^{[\omega]}_{[\V]}) = \ZZ^{[\omega]}_{[\V]}$ as locally convex spaces.
\end{theorem}	
 We end this section by discussing the structural and topological properties of the multiplier spaces of $\mathcal{S}^{[\omega]}_{[\eta]}$ and $\mathcal{S}^{[M]}_{[A]}$. Given two weight functions $\omega$ and $\eta$, we write $\ZZ^{[\omega]}_{[\eta]} = \ZZ^{[\omega]}_{[\V_\eta]}$.
 \begin{theorem}\label{UB-1}
Let $\omega$ and $\eta$ be two weight functions. Suppose that $\mathcal{S}^{[\omega]}_{[\eta]} \neq \{0\}$. Then,  $\mathcal{O}_{M}( \mathcal{S}^{[\omega]}_{[\eta]}) = \ZZ^{[\omega]}_{[\eta]}$ as locally convex spaces. Moreover,
\begin{itemize}
\item[$(i)$] $\mathcal{O}_{M}( \mathcal{S}^{(\omega)}_{(\eta)})$ is ultrabornological.
\item[$(ii)$] If $\mathcal{S}^{\{\omega\}}_{\{\eta\}}$ is Gabor accessible, then $\mathcal{O}_{M}( \mathcal{S}^{\{\omega\}}_{\{\eta\}})$ is not  ultrabornological.
\end{itemize}
\end{theorem}
\begin{proof}
This follows from Lemma \ref{cond-wf}, Theorem \ref{t:ZSpTop} and Theorem \ref{main-2}.
\end{proof}
We refer to Proposition \ref{GA-2} for conditions on $\omega$ and $\eta$ which ensure that $\mathcal{S}^{\{\omega\}}_{\{\eta\}}$ is Gabor accessible.

Let $M$ and $A$ be two weight sequences.  For $h,\lambda > 0$ we define $\mathcal{Z}^{M,h}_{A,\lambda}$ as the Banach space consisting of all $\varphi \in C^{\infty}(\R^{d})$ such that
\[ \norm{\varphi}_{\mathcal{Z}^{M,h}_{A,\lambda} } := \sup_{\alpha \in \N^{d}} \sup_{ x \in \R^{d}} \frac{|\varphi^{(\alpha)}(x)|e^{-\omega_A\left(\frac{1}{\lambda}|x| \right)}}{h^{|\alpha|}M_{|\alpha|}} < \infty.\]
We define
	\[ \mathcal{Z}^{(M)}_{(A)} := \varprojlim_{h \rightarrow 0^{+}} \varinjlim_{\lambda \rightarrow 0^+} \mathcal{Z}^{M,h}_{A,\lambda} , \qquad \mathcal{Z}^{\{M\}}_{\{A\}} := \varprojlim_{\lambda \rightarrow \infty} \varinjlim_{h \rightarrow \infty}  \mathcal{Z}^{M,h}_{A,\lambda}. \]
Then, $\ZZ^{[M]}_{[A]}$ is a $(PLB)$-space.  

\begin{lemma}\label{ZS-red}
Let $M$ and $A$ be two weight sequences. Suppose that $M$ satisfies $(M.2)$ and $(M.2)^*$. Then,
$\mathcal{Z}^{[M]}_{[A]} = \mathcal{Z}^{[\omega_M]}_{[\V_A]}$ as locally convex spaces. 
\end{lemma}
\begin{proof}
This follows from Lemma \ref{p:M2*equivBMT}. 
\end{proof}

\begin{theorem} \label{UB-2}
Let $M$ be a weight sequence satisfying $(M.2)$ and $(M.2)^*$ and let $A$ be a weight sequence satisfying $(M.2)$. Suppose that $\mathcal{S}^{[M]}_{[A]} \neq \{0\}$.
Then,  $\mathcal{O}_{M}( \mathcal{S}^{[M]}_{[A]}) = \ZZ^{[M]}_{[A]}$ as locally convex spaces. Moreover,
\begin{itemize}
\item[$(i)$] $\mathcal{O}_{M}( \mathcal{S}^{(M)}_{(A)})$ is ultrabornological.
\item[$(ii)$] If $A$ satisfies \eqref{Omega-seq}, then $\mathcal{O}_{M}( \mathcal{S}^{\{M\}}_{\{A\}})$ is ultrabornological. If $\mathcal{S}^{\{M\}}_{\{A\}}$ is Gabor accessible, the converse holds true as well.
\end{itemize} 
\end{theorem}
\begin{proof}
In view of Lemma \ref{GS-red} and Lemma \ref{ZS-red}, this follows from Lemma \ref{cond-seq}, Theorem \ref{t:ZSpTop} and Theorem \ref{main-2}.
\end{proof}
We refer to Proposition \ref{GA-3} for conditions on $M$ and $A$ which ensure that $\mathcal{S}^{\{M\}}_{\{A\}}$ is  Gabor accessible.
\begin{theorem}\label{UB-3}
Let $r,s > 0$ be such that $r+s > 1$ $(r+s \geq 1)$. Then, $\mathcal{O}_{M}( \Sigma^{r}_{s}) = \ZZ^{(p!^r)}_{(p!^s)}$  ($\mathcal{O}_{M}( \mathcal{S}^{r}_{s}) = \ZZ^{\{p!^r\}}_{\{p!^s\}}$) as locally convex spaces. Moreover,
\begin{itemize}
\item[$(i)$] $\mathcal{O}_{M}( \Sigma^{r}_{s})$ is ultrabornological.
\item[$(ii)$] If $\max(r, s) > 1$ or $\min(r, s) \geq 1/2$, then $\mathcal{O}_{M}( \mathcal{S}^{r}_{s})$ is not  ultrabornological.
\end{itemize}
\end{theorem}
\begin{proof}
This follows from  Proposition \ref{GA-3} and Theorem \ref{UB-2}.
\end{proof}
\section{Weighted $(PLB)$-spaces of continuous functions}\label{sect-cont}
Let $X$ be a topological space.  A double sequence $\A = \{ a_{N, n} ~  | ~  N,n \in \N \}$  consisting of continuous functions $a_{N,n}: X \rightarrow (0,\infty)$ is called a \emph{weight grid on X}  if  $a_{N,n+1}(x) \leq a_{N,n}(x) \leq a_{N+1,n}(x)$
for all $x \in X$ and $N,n \in \N$. Following \cite{A-B-B-ProjLimWeighLBSpContFunc}, we introduce the following two conditions:

	\begin{definition}
		A weight   grid $\A$ on $X$ is said to satisfy condition $\condQ$ if
			\begin{gather*}
				\forall N  ~ \exists M \geq N ~ \exists n ~ \forall K \geq M
				~\forall m  \geq n ~ \forall \varepsilon > 0 ~ \exists k \geq m  ~ \exists C > 0 ~ \forall x \in X ~  : \\
				\frac{1}{a_{M,m}(x)} \leq \frac{\varepsilon}{a_{N,n}(x)} + \frac{C}{a_{K,k}(x)}. 
			\end{gather*}
		If ``$\forall \varepsilon > 0$" is replaced by ``$\exists \varepsilon > 0$", then $\A$ is said to satisfy condition $\condwQ$.
	\end{definition}
For a continuous function $v : X \rightarrow (0,\infty)$ we denote by $C_v(X)$ the Banach space consisting of all $f \in C(X)$ such that $\| f \|_v = \sup_{x \in X} |f(x)| v(x) < \infty$. Given a weight grid $\A$ on $X$, we define the $(PLB)$-space
	\[ \A C(X) := \varprojlim_{N \in \N} \varinjlim_{n \in \N} C_{a_{N,n}}(X) . \]
We now give two results from \cite{A-B-B-ProjLimWeighLBSpContFunc} that will play an essential role in the proof of Theorem \ref{t:ZSpTop}.

	\begin{theorem}\cite[Theorem 3.5]{A-B-B-ProjLimWeighLBSpContFunc}
		\label{t:QImplyUltraBorn}
		Let $\A$ be a weight grid on $X$. If $\A$ satisfies $\condQ$, then $\A C(X)$ is ultrabornological.
	\end{theorem}
	
	\begin{theorem}\cite[Theorem 3.8(2)]{A-B-B-ProjLimWeighLBSpContFunc}
		\label{t:BarrelImplywQ}
		Let $\A$ be a weight  grid on $X$. If $\A C(X)$ is barrelled, then $\A$ satisfies $\condwQ$. 
	\end{theorem}

Let $X$ and $Y$ be two topological spaces. Let $\V$ be a weight system on $X$ and let $\W$ be a weight system on $Y$. We define the following weight  grids on $X \times Y$
$$
\A_{(\V,\W)} := \left \{ \frac{v_{1/N}}{w_{1/n}} ~  | ~  N,n \in \N \right\}, \qquad \A_{\{\V,\W\}} := \left \{ \frac{w_{n}}{v_{N}} ~  | ~  N,n \in \N \right\}.
$$
The following result is inspired by \cite[Theorem 4.2 and Theorem 4.3]{V-FuncFSp}.
\begin{lemma}
		\label{l:EquivWeighFuncGridCond}
		Let $\omega$ be a weight function and  let $\V$ be a weight system on a topological space $X$. Then,			\begin{itemize}
				\item[$(a)$]  The following statements are equivalent:
					\begin{itemize}
						\item[$(i)$] $\A_{(\V_\omega,\V)}$ on $\R^d \times X$ satisfies $\condQ$.
						\item[$(ii)$] $\A_{(\V_\omega,\V)}$ on $\R^d \times X$ satisfies $\condwQ$.
						\item[$(iii)$] $\V$ satisfies $(\condDN)$.
					\end{itemize}
				\item[$(b)$]  The following statements are equivalent:
				\begin{itemize}
						\item[$(i)$] $\A_{\{\V,\V_\omega\}}$ on $X \times \R^d$ satisfies $\condQ$.
						\item[$(ii)$] $\A_{\{\V,\V_\omega\}}$ on $X \times \R^d$ satisfies $\condwQ$.
						\item[$(iii)$] $\V$ satisfies $(\condooOmega)$.
					\end{itemize}					
			\end{itemize}
	\end{lemma}
\begin{proof}
We only show $(a)$ as $(b)$ can be shown similarly. The implication $(i) \Rightarrow (ii)$ is trivial. Next, we show $(ii) \Rightarrow (iii)$. Condition $\condwQ$ implies that there exists $H > 0$ such that
			\begin{gather*}
				  \exists \lambda ~ \forall \mu \leq \lambda ~ \exists \nu \leq \mu ~ \exists C > 0  ~ \forall x \in X ~ \forall t\geq 0 ~  : ~  v_{\mu}(x) \leq C\left(v_{\lambda}(x)e^{H\omega(t)} + \frac{v_{\nu}(x)}{e^{\omega(t)}}\right).
			\end{gather*}
Since $\omega(0) = 0$, $\omega$ is continuous and $\omega(t) \rightarrow \infty$ as $t \to \infty$, we obtain that  
\begin{gather*}
				  \exists \lambda ~ \forall \mu \leq \lambda ~ \exists \nu \leq \mu ~ \exists C > 0  ~ \forall x \in X ~ \forall r > 0 ~  : ~  v_{\mu}(x) \leq C\left(v_{\lambda}(x)r^H + \frac{v_{\nu}(x)}{r} \right).
			\end{gather*} 
			By calculating the minimum for $r > 0$ (with $x \in X$ fixed) of the right-hand side of the above inequality, we find that
\begin{gather*}
				 \exists \theta \in (0,1) ~  \exists \lambda ~ \forall \mu \leq \lambda ~ \exists \nu \leq \mu ~ \exists C > 0  ~  \forall x \in X ~  : ~  v_{\mu}(x) \leq Cv^{\theta}_{\lambda}(x)v^{1-\theta}_{\nu}(x).
\end{gather*}	
An induction argument now shows that $\V$ satisfies $(\condDN)$. Finally, we show $(iii) \Rightarrow (i)$. Let $N \in \N$ be arbitrary and set $M = N+1$. Since $\V$ satisfies $(\condDN)$, there is $n \in \N$ such that
\begin{equation}
\forall m \geq n  ~ \forall \theta \in (0,1) ~ \exists k \geq m  ~ \exists C > 0 ~\forall x \in X ~: ~ v_{1/m}(x) \leq Cv^\theta_{1/n}(x)v^{1-\theta}_{1/k}(x). 
\label{DN-proof}
\end{equation}
Let $K > M$, $m \geq n$ and $\varepsilon > 0$ be arbitrary. Set $\theta = (K-N-1)/(K-N) \in (0,1)$ and note that $M = \theta N + (1-\theta)K$. Choose $k$ and $C$ as in \eqref{DN-proof}. Then,
\begin{align*}
\frac{v_{1/m}(x)}{e^{M\omega(t)}} &\leq  \left(  \frac{\varepsilon v_{1/n}(x)}{e^{N\omega(t)}}  \right)^{\theta} \left(  \frac{(C\varepsilon^{-\theta})^{1/(1-\theta)}v_{1/k}(x)}{e^{K\omega(t)}}  \right)^{1-\theta} \\
&\leq \max \left \{ \varepsilon \frac{v_{1/n}(x)}{e^{N\omega(t)}} , (C\varepsilon^{-\theta})^{1/(1-\theta)} \frac{v_{1/k}(x)}{e^{K\omega(t)}} \right \} \\
&\leq   \frac{\varepsilon v_{1/n}(x)}{e^{N\omega(t)}}  +\frac{ (C\varepsilon^{-\theta})^{1/(1-\theta)} v_{1/k}(x)}{e^{K\omega(t)}},
\end{align*}
for all $t \geq 0$ and $x \in X$, whence  $\A_{(\V_\omega,\V)}$ satisfies $\condQ$.
\end{proof}	

\section{Proof of the main results}\label{sect-proofs}
The proof of part $(a)$ of Theorem \ref{t:ZSpTop} is based on the mapping properties of the STFT on  $\ZZ^{[\omega]}_{[\V]}$. We start with the following three general results:
\begin{lemma}\label{main-lemma-1} Let $\omega$ be a weight function. Let $v_i: \R^d \to (0,\infty)$, $i = 1,2,3,4$, be continuous functions such that
$$
v_2(x+t) \leq C_0v_1(x) \widetilde{v}_4(t), \qquad x,t \in \R^d,
$$
for some $C_0 > 0$ and $v_4/v_3 \in L^1$.  Let $h_i > 0$, $i= 1,2,3$, be such that
$$
\frac{1}{\max\{h_1,h_3\}} \phi^* ( \max\{h_1,h_3\}(y+1)) + (\log\sqrt{d}) y \leq \frac{1}{h_2} \phi^* (h_2y) 
+ \log C_1, \qquad y \geq 0, 
$$
for some $C_1 > 0$.  Let $\psi \in \mathcal{S}^{\omega,h_3}_{v_3}$. Then,   the mapping 
$$
V_\psi : \mathcal{S}^{\omega,h_1}_{v_1} \rightarrow C_{v_2 \otimes e^{\frac{1}{h_2}\omega}}(\R^{2d}_{x,\xi})
$$
is continuous.
\end{lemma}
\begin{proof} Let $\varphi \in \mathcal{S}^{\omega,h_1}_{v_1}$ be arbitrary. For all $y \geq 0$ and $(x, \xi) \in \R^{2d}$ with $|\xi| \geq 1$ it holds that
\begin{align*}
&|\xi|^{y} |V_\psi \varphi(x,\xi)| v_2(x) \\
&\leq \max_{|\alpha| = \lceil y \rceil} \sqrt{d}^{|\alpha|} | \xi^\alpha V_\psi \varphi(x,\xi)| v_2(x) \\
& \leq \max_{|\alpha| = \lceil y \rceil} (\sqrt{d}/2\pi)^{|\alpha|} \sum_{\beta \leq \alpha} \binom{\alpha}{\beta} v_2(x) \int_{\R^d} |\varphi^{(\beta)}(t)| |\psi^{(\alpha- \beta)}(t-x)| \dt \\
& \leq C_0\max_{|\alpha| = \lceil y \rceil} (\sqrt{d}/2\pi)^{|\alpha|} \sum_{\beta \leq \alpha}\binom{\alpha}{\beta}   \int_{\R^d} |\varphi^{(\beta)}(t)| v_1(t) |\psi^{(\alpha- \beta)}(t-x)| v_4(t-x) \dt \\
& \leq C_0 \| \varphi\|_{\mathcal{S}^{\omega,h_1}_{v_1}}  \| \psi\|_{\mathcal{S}^{\omega,h_3}_{v_3}} \| v_4/v_3\|_{L^1} \times \\
&\phantom{\leq} \max_{|\alpha| = \lceil y \rceil} (\sqrt{d}/2\pi)^{|\alpha|} \sum_{\beta \leq \alpha}\binom{\alpha}{\beta} \exp \left( \frac{1}{h_1} \phi^*(h_1|\beta|) + \frac{1}{h_3} \phi^*(h_3|\alpha -\beta|) \right)  \\
& \leq \sqrt{d}C_0 \| \varphi\|_{\mathcal{S}^{\omega,h_1}_{v_1}}  \| \psi\|_{\mathcal{S}^{\omega,h_3}_{v_3}} \| v_4/v_3\|_{L^1}  \times \\
&\phantom{\leq} \exp \left(\frac{1}{\max\{h_1,h_3\}} \phi^* ( \max\{h_1,h_3\}(y+1)) + (\log\sqrt{d}) y \right) \\
& \leq \sqrt{d}C_0C_1 \| \varphi\|_{\mathcal{S}^{\omega,h_1}_{v_1}}  \| \psi\|_{\mathcal{S}^{\omega,h_3}_{v_3}} \| v_4/v_3\|_{L^1} \exp \left(\frac{1}{h_2} \phi^* (h_2y) \right).
\end{align*}
Hence,
\begin{align*}
 |V_\psi \varphi(x,\xi)| v_2(x)  &\leq  \sqrt{d}C_0C_1 \| \varphi\|_{\mathcal{S}^{\omega,h_1}_{v_1}}  \| \psi\|_{\mathcal{S}^{\omega,h_3}_{v_3}} \| v_4/v_3\|_{L^1} \inf_{y \geq 0}  \exp \left(\frac{1}{h_2} \phi^* (h_2y) - (\log |\xi|) y  \right) \\
&= \sqrt{d}C_0C_1 \| \varphi\|_{\mathcal{S}^{\omega,h_1}_{v_1}}  \| \psi\|_{\mathcal{S}^{\omega,h_3}_{v_3}} \| v_4/v_3\|_{L^1} e^{-\frac{1}{h_2} \omega(\xi)}.
\end{align*}
For all $(x, \xi) \in \R^{2d}$ with $|\xi| \leq 1$ we have that
\begin{align*}
 |V_\psi \varphi(x,\xi)| v_2(x)e^{\frac{1}{h_2} \omega(\xi)} &\leq e^{\frac{1}{h_2}\omega(1)} v_2(x) \int_{\R^d} |\varphi(t)| |\psi(t-x)| \dt \\
& \leq C_0e^{\frac{1}{h_2}\omega(1)}  \| \varphi\|_{\mathcal{S}^{\omega,h_1}_{v_1}}  \| \psi\|_{\mathcal{S}^{\omega,h_3}_{v_3}} \| v_4/v_3\|_{L^1}.
\end{align*}
This shows that $V_\psi : \mathcal{S}^{\omega,h_1}_{v_1} \rightarrow C_{v_2 \otimes e^{\frac{1}{h_2}\omega}}(\R^{2d}_{x,\xi})$ is continuous.
\end{proof}

\begin{lemma}\label{main-lemma-2}  Let $\omega$ be a weight function. Choose $C_0,L > 0$ such that
\begin{equation}
\omega(2\pi t) \leq L \omega(t) + \log C_0, \qquad t \geq 0.
\label{in-1}
\end{equation}
Let $v_i: \R^d \to (0,\infty)$, $i = 2,3,4$, be continuous functions such that
\begin{equation}
v_2(x+t) \leq C_1v_3(x) v_4(t), \qquad x,t \in \R^d,
\label{in-2}
\end{equation}
for some $C_1 > 0$. Let $h_i > 0$, $i= 1,2,3$, be such that
\begin{equation}
\frac{1}{\max\{h_1,h_3\}} \phi^* ( \max\{h_1,h_3\}y) + (\log 2) y \leq \frac{1}{h_2} \phi^* (h_2y) 
+ \log C_2, \qquad y \geq 0, 
\label{in-3}
\end{equation}
for some $C_2 > 0$. 
Then, there is $C > 0$ such that
$$
\| M_\xi T_x \psi \|_{\mathcal{S}^{\omega,h_2}_{v_2}} \leq C \|  \psi \|_{\mathcal{S}^{\omega,h_3}_{v_3}} v_4(x) e^{\frac{L}{h_1}\omega(\xi)}, \qquad (x,\xi) \in \R^{2d},
$$
for all $\psi \in \mathcal{S}^{\omega,h_3}_{v_3}$.
\end{lemma}
\begin{proof}
Let $\psi \in \mathcal{S}^{\omega,h_3}_{v_3}$ and $(x,\xi) \in \R^{2d}$ be arbitrary. For all $\alpha \in \N^d$ and $t \in \R^d$ it holds that
\begin{align*}
|(M_\xi T_x \psi)^{(\alpha)}(t)| v_2(t) &\leq  \sum_{\beta \leq \alpha} \binom{\alpha}{\beta} (2\pi |\xi|)^{|\beta|} |\psi^{(\alpha-\beta)} (t-x)| v_2(t) \\
&\leq C_1v_4(x)  \sum_{\beta \leq \alpha} \binom{\alpha}{\beta} (2\pi |\xi|)^{|\beta|}  |\psi^{(\alpha-\beta)} (t-x)| v_3(t-x) \\
&\leq C_1\|\psi\|_{\mathcal{S}^{\omega,h_3}_{v_3}} v_4(x) \sum_{\beta \leq \alpha} \binom{\alpha}{\beta}   \exp \left( (\log 2\pi |\xi|)|\beta| - \frac{1}{h_1} \phi^*(h_1 |\beta|) \right) \times \\
& \phantom{\leq}  \exp \left(\frac{1}{h_3} \phi^*(h_3 |\alpha-\beta|) + \frac{1}{h_1} \phi^*(h_1 |\beta|) \right) \\
&\leq  C_1\|\psi\|_{\mathcal{S}^{\omega,h_3}_{v_3}} v_4(x) e^{\frac{1}{h_1} \omega(2\pi \xi)}    \times \\
& \phantom{\leq} \exp \left(\frac{1}{\max\{h_1,h_3\}} \phi^*(\max\{h_1,h_3\} |\alpha|) + (\log 2)|\alpha| \right) \\
& \leq C_0C_1C_2\|\psi\|_{\mathcal{S}^{\omega,h_3}_{v_3}} v_4(x) e^{\frac{L}{h_1} \omega(\xi)}  \exp \left(\frac{1}{h_2} \phi^*(h_2 |\alpha|) \right).
\end{align*}
This shows the result.
\end{proof}
\begin{lemma}\label{main-lemma-20}  Let $\omega$ be a weight function. Choose $C_0,L > 0$ such that \eqref{in-1} holds.
Let $v_i: \R^d \to (0,\infty)$, $i = 1,2,3,4$, be continuous functions such that \eqref{in-2} holds for some $C_1 > 0$ and $v_4/v_1 \in L^1$.
Let $h_i > 0$, $i= 1,2,3$, be such that  \eqref{in-3} holds for some $C_2 > 0$.  Let $\gamma \in \mathcal{S}^{\omega,h_3}_{v_3}$. Then,   the mapping 
$$
V^*_\gamma : C_{v_1 \otimes e^{\frac{2L}{h_1}\omega}}(\R^{2d}_{x,\xi}) \rightarrow \mathcal{S}^{\omega,h_2}_{v_2} 
$$
is continuous.
\end{lemma}
\begin{proof}
This follows from Lemma \ref{main-lemma-2}.
\end{proof}
Given a weight function $\omega$ and a weight system $\V$, we define
\begin{gather*}
C_{\ZZ^{(\omega)}_{(\V)}}(\R^{2d}_{x,\xi}) := \varprojlim_{h \to 0^+}\varinjlim_{\lambda \to 0^+} C_{ \frac{1}{v_\lambda} \otimes e^{\frac{1}{h}\omega}}(\R^{2d}_{x,\xi}), \\
C_{\ZZ^{\{\omega\}}_{\{\V\}}}(\R^{2d}_{x,\xi}) := \varprojlim_{\lambda \to \infty}\varinjlim_{h \to \infty} C_{ \frac{1}{v_\lambda} \otimes e^{\frac{1}{h}\omega}}(\R^{2d}_{x,\xi}).
\end{gather*}
We then have:
\begin{proposition} \label{STFT-Z}
Let $\omega$ be a weight function and let $\V$ be a weight system satisfying $[\condM]$ and $[\condN]$. Let $\psi \in \mathcal{S}^{[\omega]}_{[\V]}$ and $\gamma \in \mathcal{S}^{[\omega]}_{[ \widetilde \V]}$. Then, the mappings
$$
V_\psi: \ZZ^{[\omega]}_{[\V]} \rightarrow C_{\ZZ^{[\omega]}_{[\V]}}(\R^{2d}_{x,\xi}), \qquad
V^*_\gamma:  C_{\ZZ^{[\omega]}_{[\V]}}(\R^{2d}_{x,\xi}) \rightarrow \ZZ^{[\omega]}_{[\V]}
$$
are continuous. Moreover, if $(\gamma, \psi)_{L^2} \neq 0$, then
\begin{equation}
\frac{1}{(\gamma, \psi)_{L^{2}}} V^{*}_{\gamma} \circ V_{\psi} = \id_{\ZZ^{[\omega]}_{[\V]}}.
\label{eq:STFTZSpReconstruct} 
\end{equation}
\end{proposition}
\begin{proof}
It suffices to show that
$$
\forall h ~ \exists k ~ \forall \lambda ~\exists \mu ~ (\forall \mu ~ \exists \lambda ~ \forall k ~\exists h) ~  : ~  V_\psi: \mathcal{S}^{\omega,k}_{1/v_\lambda} \rightarrow  C_{ \frac{1}{v_\mu} \otimes e^{\frac{1}{h}\omega}}(\R^{2d}_{x,\xi}) \mbox{ is continuous,}
$$
and 
$$
\forall h ~ \exists k ~ \forall \lambda ~\exists \mu ~ (\forall \mu ~ \exists \lambda ~ \forall k ~\exists h) ~  : ~  V^*_\gamma: C_{ \frac{1}{v_\lambda} \otimes e^{\frac{1}{k}\omega}}(\R^{2d}_{x,\xi}) \rightarrow \mathcal{S}^{\omega,h}_{1/v_\mu}    \mbox{ is continuous.}
$$
By Lemma \ref{M12} and the fact that $\V$ satisfies $[\condM]$ and $[\condN]$, this follows from Lemma \ref{main-lemma-1} and Lemma \ref{main-lemma-20}, respectively.
We now show \eqref{eq:STFTZSpReconstruct}. Let $\varphi \in \ZZ^{[\omega]}_{[\V]}$ be arbitrary. Since $\V$ satisfies $[\condN]$ and $\omega$ satisfies $(\gamma)$, the continuous functions $\varphi T_{x} \overline{\psi}$ and $V_\psi \varphi (x, \, \cdot \,)$, with $x \in \R^d$ fixed, both belong to $L^1$. As $V_\psi \varphi (x, \xi)  = \mathcal{F}( \varphi T_{x} \overline{\psi})(\xi)$, we obtain that
\begin{align*}
&  \int \int_{\R^{2d}} V_\psi \varphi (x, \xi) M_\xi T_x \gamma(t) \dx \dxi =  \int_{\R^d} \left( \int_{\R^d} V_\psi \varphi (x, \xi) e^{2\pi i \xi \cdot t} d\xi \right) T_x \gamma(t)  \dx \\
&= \varphi(t) \int_{\R^d}  T_{x} \overline{\psi}(t)  T_x\gamma(t)  \dx  = (\gamma, \psi)_{L^2} \varphi(t)
\end{align*}
for all $t \in \R^d$.
\end{proof}
\begin{lemma}\label{ex-window}
Let $\omega$ be a weight function and let $\V$ be a weight system satisfying $[\condwM]$. If $\mathcal{S}^{[\omega]}_{[\V]} \neq \{0\}$, then also $\mathcal{S}^{[\omega]}_{[\V]} \cap \mathcal{S}^{[\omega]}_{[\widetilde \V]} \neq \{0\}$.
\end{lemma}
\begin{proof}
Since the space  $\mathcal{S}^{[\omega]}_{[\V]}$ is translation-invariant and non-trivial, there exists $\varphi \in  \mathcal{S}^{[\omega]}_{[\V]}$ with $\varphi(0) \neq 0$. Then, $\psi = \varphi \widetilde{\varphi} \in \mathcal{S}^{[\omega]}_{[\V]} \cap \mathcal{S}^{[\omega]}_{[\widetilde \V]}$ and $\psi(0) = \varphi^2(0) \neq 0$.
\end{proof}

\begin{proof}[Proof of part $(a)$ of Theorem \ref{t:ZSpTop}] The implication $(ii) \Rightarrow (iii)$ holds for any locally convex space. We now show $(i) \Rightarrow (ii)$. By  Lemma \ref{ex-window}, there exists $\psi \in \mathcal{S}^{[\omega]}_{[\V]} \cap \mathcal{S}^{[\omega]}_{[\widetilde \V]} \backslash \{ 0\}$.  Proposition \ref{STFT-Z} (with $\gamma = \psi $) implies that $\ZZ^{[\omega]}_{[\V]}$ is isomorphic to a complemented subspace of $C_{\ZZ^{[\omega]}_{[\V]}}(\R^{2d}_{x,\xi})$. Note that
$$
C_{\ZZ^{(\omega)}_{(\V)}}(\R^{2d}_{x,\xi}) = \A_{(\V_\omega,\V)}C(\R^{2d}_{\xi,x}), \qquad C_{\ZZ^{\{\omega\}}_{\{\V\}}}(\R^{2d}_{x,\xi}) = \A_{\{\V,\V_\omega\}}C(\R^{2d}_{x,\xi}).
$$
Hence, $C_{\ZZ^{[\omega]}_{[\V]}}(\R^{2d}_{x,\xi})$ is ultrabornological by Theorem \ref{t:QImplyUltraBorn} and Lemma \ref{l:EquivWeighFuncGridCond}. The result now follows from the fact that a complemented subspace of an ultrabornological space is again ultrabornological.
\end{proof}
The proof of part $(b)$ of Theorem \ref{t:ZSpTop} is based on the mapping properties of the analysis and synthesis operator on $\ZZ^{[\omega]}_{[\V]}$. Given a weight function $\omega$, a weight system $\V$ and $a,b > 0$, we define
\begin{gather*}
C_{\ZZ^{(\omega)}_{(\V)}}(a\Z^d_x \times b\Z^d_\xi) := \varprojlim_{h \to 0^+}\varinjlim_{\lambda \to 0^+} C_{ \frac{1}{v_\lambda} \otimes e^{\frac{1}{h}\omega}}(a\Z^d_x \times b\Z^d_\xi), \\
C_{\ZZ^{\{\omega\}}_{\{\V\}}}(a\Z^d_x \times b\Z^d_\xi) := \varprojlim_{\lambda \to \infty}\varinjlim_{h \to \infty} C_{ \frac{1}{v_\lambda} \otimes e^{\frac{1}{h}\omega}}(a\Z^d_x \times b\Z^d_\xi).
\end{gather*}
\begin{proposition}\label{GF-Z}
Let $\omega$ be a weight function, let $\V$ be a weight system satisfying $[\condM]$ and $[\condN]$ and let $a,b>0$. Let $\psi \in \mathcal{S}^{[\omega]}_{[\V]}$ and $\gamma \in \mathcal{S}^{[\omega]}_{[ \widetilde \V]}$. Then, the mappings
$$
C^{a,b}_\psi: \ZZ^{[\omega]}_{[\V]} \rightarrow C_{\ZZ^{[\omega]}_{[\V]}}(a\Z^d_x \times b\Z^d_\xi) , \qquad
D^{a,b}_\gamma:  C_{\ZZ^{[\omega]}_{[\V]}}(a\Z^d_x \times b\Z^d_\xi)  \rightarrow \ZZ^{[\omega]}_{[\V]}
$$
are continuous. Moreover, if $(\psi, \gamma)$ is a dual pair of windows on $a\Z^d \times b\Z^d$, then
\begin{equation}
\frac{1}{(ab)^d} C^{1/a,1/b}_\psi \circ D^{1/a,1/b}_\gamma = \operatorname{id}_{C_{\ZZ^{[\omega]}_{[\V]}}\left( \frac{1}{a}\Z^d_x \times \frac{1}{b}\Z^d_\xi\right)}.
\label{reverse}
\end{equation}
\end{proposition}
\begin{proof}
It suffices to show that
$$
\forall h ~ \exists k ~ \forall \lambda ~\exists \mu  ~ (\forall \mu ~ \exists \lambda ~ \forall k ~\exists h) ~  : ~  C^{a,b}_\psi: \mathcal{S}^{\omega,k}_{1/v_\lambda} \rightarrow  C_{ \frac{1}{v_\mu} \otimes e^{\frac{1}{h}\omega}}(a\Z^d_x \times b\Z^d_\xi) \mbox{ is continuous,}
$$
and 
$$
\forall h ~ \exists k ~ \forall \lambda ~\exists \mu  ~ (\forall \mu ~ \exists \lambda ~ \forall k ~\exists h) ~  : ~  D^{a,b}_\gamma: C_{ \frac{1}{v_\lambda} \otimes e^{\frac{1}{k}\omega}}(a\Z^d_x \times b\Z^d_\xi) \rightarrow \mathcal{S}^{\omega,h}_{1/v_\mu}    \mbox{ is continuous.}
$$
By Lemma \ref{M12} and the fact that $\V$ satisfies $[\condM]$ and $[\condN]$, the first statement follows from Lemma \ref{main-lemma-1} while the second statement follows from Lemma \ref{main-lemma-2} together with \eqref{discrete-l1} (cf.\ Lemma \ref{main-lemma-20}). Since, by \eqref{C0-wfs}, the set of finite sequences on $a\Z^d_x \times b\Z^d_\xi$ is dense in $C_{\ZZ^{[\omega]}_{[\V]}}(a\Z^d_x \times b\Z^d_\xi)$,  the identity \eqref{reverse} follows from Lemma \ref{l:WZBiOrthRel}.
\end{proof}
\begin{proof}[Proof of part $(b)$ of Theorem \ref{t:ZSpTop}] As $\mathcal{S}^{[\omega]}_{[\V]}$ is Gabor accessible, Proposition \ref{GF-Z}  implies that there are $a,b >0$ such that $C_{\ZZ^{[\omega]}_{[\V]}}\left( \frac{1}{a}\Z^d_x \times \frac{1}{b}\Z^d_\xi\right)$ is isomorphic to a complemented subspace of $\ZZ^{[\omega]}_{[\V]}$. Since a complemented subspace of a barrelled space is again barrelled, we may conclude that $C_{\ZZ^{[\omega]}_{[\V]}}\left( \frac{1}{a}\Z^d_x \times \frac{1}{b}\Z^d_\xi\right)$ is barrelled. Note that 
\begin{gather*}
C_{\ZZ^{(\omega)}_{(\V)}}\left( \frac{1}{a}\Z^d_x \times \frac{1}{b}\Z^d_\xi\right)= \A_{(\V_\omega \mid \frac{1}{b}\Z^d ,\V | \frac{1}{a}\Z^d)}C\left( \frac{1}{b}\Z^d_\xi \times \frac{1}{a}\Z^d_x  \right), \\
C_{\ZZ^{\{\omega\}}_{\{\V\}}}\left( \frac{1}{a}\Z^d_x \times \frac{1}{b}\Z^d_\xi\right) = \A_{\{\V | \frac{1}{a}\Z^d,\V_\omega \mid \frac{1}{b}\Z^d\}}C\left( \frac{1}{a}\Z^d_x \times \frac{1}{b}\Z^d_\xi \right).
\end{gather*}
Hence, $\A_{(\V_\omega \mid \frac{1}{b}\Z^d ,\V | \frac{1}{a}\Z^d)}$ ($\A_{\{\V | \frac{1}{a}\Z^d,\V_\omega \mid \frac{1}{b}\Z^d\}}$) satisfies $(wQ)$ by Theorem \ref{t:BarrelImplywQ}. Properties \eqref{alpha-extended} and \eqref{joo}  imply that then also $\A_{(\V_\omega ,\V)}$ ($\A_{\{\V,\V_\omega \}}$) satisfies $(wQ)$. The result now follows from  Lemma \ref{l:EquivWeighFuncGridCond}.
\end{proof}
Finally, we show Theorem \ref{main-2}. We need various results in preparation.

\begin{lemma}\label{main-lemma-3} Let $\omega$ be a weight function. Let $v_i: \R^d \to (0,\infty)$, $i = 1,2,3$, be continuous functions such that
$$
v_2(x) \leq C_0v_1(x) v_3(x), \qquad x \in \R^d,
$$
for some $C_0 > 0$. Let $h_i > 0$, $i= 1,2,3$, be such that
$$
\frac{1}{\max\{h_1,h_3\}} \phi^* ( \max\{h_1,h_3\}y) + (\log2) y \leq \frac{1}{h_2} \phi^* (h_2y) 
+ \log C_1, \qquad y \geq 0, 
$$
for some $C_1 > 0$. Then,  the mapping 
$$
\mathcal{S}^{\omega,h_1}_{v_1} \times  \mathcal{S}^{\omega,h_3}_{v_3} \rightarrow  \mathcal{S}^{\omega,h_2}_{v_2}, ~  (\varphi, \psi) \mapsto \varphi \cdot \psi
$$
is continuous.
\end{lemma}
\begin{proof}
Let $\varphi \in \mathcal{S}^{\omega,h_1}_{v_1}$ and $\psi \in   \mathcal{S}^{\omega,h_3}_{v_3}$ be arbitrary. For all $\alpha \in \N^d$ and $x \in \R^d$ it holds that
\begin{align*}
&|(\varphi\psi)^{(\alpha)}(x)|v_2(x) \\
&\leq C_0\sum_{\beta \leq \alpha} \binom{\alpha}{\beta} |\varphi^{(\beta)}(x)| v_1(x) |\psi^{(\alpha-\beta)}(x)| v_3(x) \\
&\leq C_0\| \varphi \|_{\mathcal{S}^{\omega,h_1}_{v_1}}\| \psi \|_{\mathcal{S}^{\omega,h_3}_{v_3}} \sum_{\beta \leq \alpha} \binom{\alpha}{\beta} \exp \left( \frac{1}{h_1} \phi^*(h_1|\beta|) + \frac{1}{h_3} \phi^*(h_3|\alpha -\beta|) \right)  \\
&\leq  C_0\| \varphi \|_{\mathcal{S}^{\omega,h_1}_{v_1}}\| \psi \|_{\mathcal{S}^{\omega,h_3}_{v_3}} \exp \left(  \frac{1}{\max\{h_1,h_3\}} \phi^* ( \max\{h_1,h_3\}|\alpha|) + (\log2) |\alpha| \right) \\
&\leq C_0C_1 \| \varphi \|_{\mathcal{S}^{\omega,h_1}_{v_1}}\| \psi \|_{\mathcal{S}^{\omega,h_3}_{v_3}} \exp \left(  \frac{1}{h_2} \phi^* (h_2|\alpha|) \right).
\end{align*}
This shows that the mapping $\mathcal{S}^{\omega,h_1}_{v_1} \times  \mathcal{S}^{\omega,h_3}_{v_3} \rightarrow  \mathcal{S}^{\omega,h_2}_{v_2}, ~  (\varphi, \psi) \mapsto \varphi \cdot \psi$ is continuous.
\end{proof}
Next, we extend the STFT and its adjoint to $\mathcal{S}'^{[\omega]}_{[\V]}$.
Given  a weight function $\omega$ and a weight system  $\V$  satisfying $[\condM]$ and $[\condN]$, we define
\begin{gather*}
C_{\mathcal{S}'^{(\omega)}_{(\V)}}(\R^{2d}_{x,\xi}) := \varinjlim_{h \to 0^+} C_{ \frac{1}{v_h} \otimes e^{-\frac{1}{h}\omega}}(\R^{2d}_{x,\xi}), \\
C_{\mathcal{S}'^{\{\omega\}}_{\{\V\}}}(\R^{2d}_{x,\xi}) :=  \varprojlim_{h \to \infty} C_{ \frac{1}{v_h} \otimes e^{-\frac{1}{h}\omega}}(\R^{2d}_{x,\xi}).
\end{gather*}
The STFT of an element $f \in  \mathcal{S}'^{[\omega]}_{[\V]}$ with respect to  $\psi \in  \mathcal{S}^{[\omega]}_{[\V]}$ is defined as
\begin{equation*}
V_\psi f(x,\xi):= \langle f, \overline{M_\xi T_x\psi}\rangle, \qquad (x, \xi) \in \R^{2d}.
\end{equation*}
We define the adjoint STFT of $F \in C_{\mathcal{S}'^{[\omega]}_{[\V]}}(\R^{2d}_{x,\xi})$ with respect to  $\gamma \in \mathcal{S}^{[\omega]}_{[\widetilde \V]}$ as
$$
\langle V^\ast_\gamma F, \varphi \rangle := \langle F, \overline{V_\gamma\overline{\varphi}} \rangle = \int \int_{\R^{2d}} F(x,\xi) V_{\overline{\gamma}}\varphi(x, -\xi) \dx \dxi, \qquad \varphi \in \mathcal{S}^{[\omega]}_{[\V]}.
$$
\begin{proposition} \label{STFT-dual}
Let $\omega$ be a weight function and let $\V$ be a weight system satisfying $[\condM]$ and $[\condN]$. Let $\psi \in \mathcal{S}^{[\omega]}_{[\V]}$ and $\gamma \in \mathcal{S}^{[\omega]}_{[\widetilde \V]}$. Then, the mappings
$$
V_\psi: \mathcal{S}'^{[\omega]}_{[\V]} \rightarrow C_{\mathcal{S}'^{[\omega]}_{[\V]}}(\R^{2d}_{x,\xi}), \qquad
V^*_\gamma:  C_{\mathcal{S}'^{[\omega]}_{[\V]}}(\R^{2d}_{x,\xi}) \rightarrow \mathcal{S}'^{[\omega]}_{[\V]}
$$
are continuous. Moreover, if $(\gamma, \psi)_{L^2} \neq 0$, then
\begin{equation}
\label{eq:reconstructSTFTTempUltra}
\frac{1}{(\gamma, \psi)_{L^{2}}} V^{*}_{\gamma} \circ V_{\psi} = \id_{\mathcal{S}'^{[\omega]}_{[\V]}}.
\end{equation}
\end{proposition}
\begin{proof}
By Lemma \ref{M12} and the fact that $\V$ satisfies $[\condM]$, Lemma \ref{main-lemma-2} implies that $V_\psi: \mathcal{S}'^{\{\omega\}}_{\{\V\}} \rightarrow C_{\mathcal{S}'^{\{\omega\}}_{\{\V\}}}(\R^{2d}_{x,\xi})$ is continuous  and that  $V_\psi: \mathcal{S}'^{(\omega)}_{(\V)} \rightarrow C_{\mathcal{S}'^{(\omega)}_{(\V)}}(\R^{2d}_{x,\xi})$ maps bounded sets into bounded sets. Since $\mathcal{S}'^{(\omega)}_{(\V)}$ is bornological, we obtain that  also $V_\psi: \mathcal{S}'^{(\omega)}_{(\V)} \rightarrow C_{\mathcal{S}'^{(\omega)}_{(\V)}}(\R^{2d}_{x,\xi})$ is continuous. 
Next, we treat $V^*_\gamma$. We define
$$
C_{\mathcal{S}^{(\omega)}_{(\V)}}(\R^{2d}_{x,\xi}) := \varprojlim_{h \to 0^+} C_{ v_h \otimes e^{\frac{1}{h}\omega}}(\R^{2d}_{x,\xi}), \qquad C_{\mathcal{S}^{\{\omega\}}_{\{\V\}}}(\R^{2d}_{x,\xi}) :=  \varinjlim_{h \to \infty} C_{ v_h \otimes e^{\frac{1}{h}\omega}}(\R^{2d}_{x,\xi}).
$$
We claim that the mapping
$$
\mathcal{S}^{[\omega]}_{[\V]} \rightarrow C_{\mathcal{S}^{[\omega]}_{[\V]}}(\R^{2d}_{x,\xi}), ~ \varphi \mapsto V_{\overline{\gamma}}\varphi(x, -\xi),
$$
is continuous.  It suffices to show that
$$
\forall h ~ \exists k ~ (\forall k ~\exists h) ~  : ~  \mathcal{S}^{\omega,k}_{v_k} \rightarrow  C_{ v_h \otimes e^{\frac{1}{h}\omega}}(\R^{2d}_{x,\xi}),  ~ \varphi \mapsto V_{\overline{\gamma}}\varphi(x, -\xi) \mbox{ is continuous.}
$$
By Lemma \ref{M12} and the fact that $\V$ satisfies $[\condM]$ and $[\condN]$, this follows from Lemma \ref{main-lemma-1}. Hence, the continuity of $V^{*}_{\gamma} : C_{\mathcal{S}^{\prime [\omega]}_{[\V]}}(\R^{2d}_{x, \xi}) \rightarrow \mathcal{S}^{\prime [\omega]}_{[\V]}$ is a consequence of the continuity of the mapping
$$
C_{\mathcal{S}'^{[\omega]}_{[\V]}}(\R^{2d}_{x,\xi}) \rightarrow (C_{\mathcal{S}^{[\omega]}_{[\V]}}(\R^{2d}_{x,\xi}))'_b, ~  F \mapsto \left ( f \mapsto \int \int_{\R^{2d}} F(x,\xi) f(x,\xi) \dx \dxi \right ).
$$

Finally, we show \eqref{eq:reconstructSTFTTempUltra}. The Hahn-Banach theorem and the fact that  $\mathcal{S}^{[\omega]}_{[\V]}$ is reflexive (as it is nuclear \cite[Theorem 5.1]{D-N-V-NuclGSKerThm}), imply that $L^2$ is dense in  $\mathcal{S}^{\prime [\omega]}_{[\V]}$. Hence, \eqref{eq:reconstructSTFTTempUltra} follows from \eqref{eq:reconstructSTFTL2}.
\end{proof}
We also need the following result about the projective description of $C_{\mathcal{Z}^{[\omega]}_{[\V]}}(\R^{2d}_{x,\xi})$.
\begin{lemma}\label{projective}
\mbox{}
\begin{itemize}
\item[$(i)$]   Let $\V$ be a weight system satisfying $[\condwM]$ and $[\condN]$. Set
$$
\overline{V}(\mathcal{V}) := \{ v : \R^{d} \rightarrow (0,\infty) \mbox{ continuous} ~  | ~  \sup_{x \in \R^d} v(x)v_\lambda(x) < \infty \mbox{ for all $\lambda > 0$} \}.
$$
Then,
$$
C_{\mathcal{Z}^{(\omega)}_{(\V)}}(\R^{2d}_{x,\xi}) =  \varprojlim_{h \to 0^+} \varprojlim_{v \in \overline{V}(\mathcal{V}) } C_{ v \otimes e^{\frac{1}{h}\omega}}(\R^{2d}_{x,\xi}).
$$
\item[$(ii)$]  Let $\omega$ be a weight function. Set
$$
\overline{V}_\omega := \{ \sigma : [0,\infty) \rightarrow [0,\infty) \mbox{ weight function} ~   | ~  \sigma= o(\omega) \}.
$$
Then,
$$
C_{\mathcal{Z}^{\{\omega\}}_{\{\V\}}}(\R^{2d}_{x,\xi}) =  \varprojlim_{\lambda \to \infty} \varprojlim_{\sigma \in \overline{V}_\omega} C_{ \frac{1}{v_\lambda} \otimes e^{\sigma}}(\R^{2d}_{x,\xi}).
$$
\end{itemize}
\end{lemma}
\begin{proof} In view of \cite[Theorem 1.3]{B-M-S}, $(i)$ follows from \eqref{C0-wfs} and \cite[Proposition, p.\ 112, and its proof]{B-M-S} and $(ii)$ follows from \cite[Lemma 1.7 and Remark 1.8]{B-M-T-UltraDiffFuncFAnal}.
\end{proof}
\begin{proof}[Proof of Theorem \ref{main-2}]
We first show that $\ZZ^{[\omega]}_{[\V]}$ is continuously included in $\mathcal{O}_{M}(\mathcal{S}^{[\omega]}_{[\V]})$. To this end, it suffices to show that
\begin{gather*}
\forall h ~ \exists k ~\forall \lambda ~ \exists \mu ~  ( \forall \mu ~ \exists \lambda ~\forall k ~ \exists h ) ~  : ~  
\mathcal{S}^{\omega,\mu}_{v_\mu} \times  \mathcal{S}^{\omega,k}_{1/v_\lambda} \rightarrow  \mathcal{S}^{\omega,h}_{v_h}, ~  (\varphi, \psi) \mapsto \varphi \cdot \psi  \mbox { is continuous.}
\end{gather*}
By Lemma \ref{M12} and the fact that $\V$ satisfies $[\condSquare]$, this follows from Lemma \ref{main-lemma-3}. Next, we show that $\mathcal{O}_{M}(\mathcal{S}^{[\omega]}_{[\V]})$ is continuously included in $\ZZ^{[\omega]}_{[\V]}$.  By  Lemma \ref{ex-window}, there exists $\chi \in \mathcal{S}^{[\omega]}_{[\V]} \cap \mathcal{S}^{[\omega]}_{[\widetilde \V]} \backslash \{0\}$. Set $\psi = \chi^2 \in \mathcal{S}^{[\omega]}_{[\V]} \cap \mathcal{S}^{[\omega]}_{[\widetilde \V]} \backslash \{0\}$. Proposition \ref{STFT-Z}  and the reconstruction formula \eqref{eq:reconstructSTFTTempUltra} (with $\psi = \gamma$) imply that it suffices to show that the mapping
$$
V_\psi: \mathcal{O}_{M}(\mathcal{S}^{[\omega]}_{[\V]}) \rightarrow C_{ \ZZ^{[\omega]}_{[\V]}}(\R^{2d}_{x,\xi})
$$
is continuous. We start by showing that 
$$
V_\psi f(x,\xi) = \int_{\R^d} ((T_x \overline{\psi}) \cdot f)(t) e^{-2\pi i \xi \cdot t} \dt, \qquad (x, \xi) \in \R^{2d},
$$
for all $f \in \mathcal{O}_{M}(\mathcal{S}^{[\omega]}_{[\V]})$. As $\psi = \chi^2$, we have that
\begin{align*}
&V_\psi f(x,\xi) = \langle f, \overline{M_\xi T_x \psi} \rangle =   \langle f, (\overline{M_\xi T_x \chi}) \cdot (T_x \overline{\chi}) \rangle =    \langle (T_x \overline{\chi} ) \cdot f, (\overline{M_\xi T_x \chi})  \rangle  \\
&= \int_{\R^d} ((T_x \overline{\chi}) \cdot f)(t) T_x\overline{\chi}(t) e^{-2\pi i \xi \cdot t} \dt =
\int_{\R^d} ((T_x \overline{\psi}) \cdot f)(t) e^{-2\pi i \xi \cdot t} \dt 
\end{align*}
for all $(x,\xi) \in \R^{2d}$. In the rest of the proof we treat the Beurling and Roumieu case separately. We first consider the Beurling case. By Lemma \ref{projective}$(i)$, it suffices to show that for all $h > 0$ and $v \in \overline{V}(\mathcal{V})$ there is a bounded subset $B \subset \mathcal{S}^{(\omega)}_{(\V)}$, $\mu > 0$ and $k, C > 0$ such that
\begin{equation}
\|V_\psi f\|_{v \otimes e^{\frac{1}{h}\omega}} \leq C \sup_{\varphi \in B} \| \varphi \cdot f \|_{ \mathcal{S}^{\omega,k}_{v_\mu}}, \qquad f \in  \mathcal{O}_M(\mathcal{S}^{(\omega)}_{(\V)}).
\label{Beurling-todo}
\end{equation}
Set 
$$
B = \{ T_x \psi v(x) ~  | ~  x \in \R^d \}.
$$
We claim that  $B$ is a bounded subset of $\mathcal{S}^{(\omega)}_{(\V)}$. Let $h, \lambda > 0$ be arbitrary. Condition $(\condM)$ yields that there are $\kappa, \nu \leq \lambda$ and $C > 0$ such that $v_\lambda(y) \leq Cv_\kappa(y-x) v_\nu(x)$ for all $x,y \in \R^d$. Hence,
\begin{align*}
\sup_{x \in \R^d} \norm{T_x \psi v(x)}_{\mathcal{S}^{\omega,h}_{v_\lambda} } &= \sup_{x\in \R^d}\sup_{\alpha \in \N^{d}} \sup_{ y \in \R^{d}} |\psi^{(\alpha)}(y-x)|v_\lambda(y)\exp \left(-\frac{1}{h}\phi^*(h|\alpha|)\right) v(x) \\
&\leq C\norm{\psi}_{\mathcal{S}^{\omega,h}_{v_\kappa} } \sup_{x \in \R^d} v(x) v_\nu(x) < \infty.
\end{align*}
Next, choose $\mu >0$ such that $1/v_\mu \in L^1$ (as $v_\lambda \geq 1$ for all $\lambda > 0$ this is possible by $(\condN)$). By Lemma \ref{M12}$(i)$, there are $k,C > 0$ such that
$$
\frac{1}{k} \phi^* ( k(y+1)) + (\log\sqrt{d}) y \leq \frac{1}{h} \phi^* (hy) 
+ \log C, \qquad y \geq 0.
$$
For all $y \geq 0$ and $(x, \xi) \in \R^{2d}$ with $|\xi| \geq 1$ it holds that
\begin{align*}
&|\xi|^{y} |V_\psi f(x,\xi)| v(x) \\
&\leq \max_{|\alpha| = \lceil y \rceil} \sqrt{d}^{|\alpha|} v(x)\left | \xi^\alpha  \int_{\R^d} ((T_x \overline{\psi}) \cdot f)(t) e^{-2\pi i \xi \cdot t} \dt \right|  \\
& \leq \max_{|\alpha| = \lceil y \rceil} (\sqrt{d}/(2\pi))^{|\alpha|}  v(x) \int_{\R^d} |((T_x \overline{\psi}) \cdot f)^{(\alpha)}(t)| \dt  \\ 
& \leq \| 1/v_\mu \|_{L^1} \sup_{\varphi \in B} \| \varphi \cdot f \|_{ \mathcal{S}^{\omega,k}_{v_\mu}} \max_{|\alpha| = \lceil y \rceil} (\sqrt{d}/(2\pi))^{|\alpha|} \exp \left( \frac{1}{k} \phi^*(k|\alpha|) \right) \\
& \leq \sqrt{d} \| 1/v_\mu \|_{L^1} \sup_{\varphi \in B} \| \varphi \cdot f \|_{ \mathcal{S}^{\omega,k}_{v_\mu}}  \exp \left( \frac{1}{k} \phi^*(k(y+1)) + (\log\sqrt{d}) y \right) \\
& \leq \sqrt{d}C \| 1/v_\mu \|_{L^1} \sup_{\varphi \in B} \| \varphi \cdot f \|_{ \mathcal{S}^{\omega,k}_{v_\mu}}  \exp \left( \frac{1}{h} \phi^*(hy) \right).
\end{align*}
Hence,
\begin{align*}
 |V_\psi f(x,\xi)| v(x)  &\leq  \sqrt{d}C \| 1/v_\mu \|_{L^1} \sup_{\varphi \in B} \| \varphi \cdot f \|_{ \mathcal{S}^{\omega,k}_{v_\mu}} \inf_{y \geq 0}  \exp \left(\frac{1}{h} \phi^* (hy) - (\log |\xi|) y  \right) \\
&= \sqrt{d}C \| 1/v_\mu \|_{L^1} \sup_{\varphi \in B} \| \varphi \cdot f \|_{ \mathcal{S}^{\omega,k}_{v_\mu}} e^{-\frac{1}{h} \omega(\xi)}.
\end{align*}
For all $(x, \xi) \in \R^{2d}$ with $|\xi| \leq 1$ it holds that
\begin{align*}
 |V_\psi f(x,\xi)| v(x)e^{\frac{1}{h} \omega(\xi)} &\leq e^{\frac{1}{h} \omega(1)} v(x)  \int_{\R^d} |((T_x \overline{\psi}) \cdot f)(t)| \dt  \\
& \leq e^{\frac{1}{h}\omega(1)}   \| 1/v_\mu \|_{L^1} \| \sup_{\varphi \in B} \| \varphi \cdot f \|_{ \mathcal{S}^{\omega,k}_{v_\mu}}.
\end{align*}
This shows \eqref{Beurling-todo}. Next, we consider the Roumieu case.  By Lemma \ref{projective}$(ii)$, it suffices to show that for all $\lambda > 0$ and $\sigma \in \overline{V}_\omega$ there is a bounded subset $B \subset \mathcal{S}^{\{\omega\}}_{\{\V\}}$, a continuous seminorm $p$ on $\mathcal{S}^{\{\omega\}}_{\{\V\}}$ and  $C > 0$ such that
\begin{equation}
\|V_\psi f\|_{\frac{1}{v_\lambda} \otimes e^{\sigma}} \leq C \sup_{\varphi \in B} p(\varphi \cdot f), \qquad f \in \mathcal{O}_M(\mathcal{S}^{\{\omega\}}_{\{\V\}}).
\label{Roumieu-todo}
\end{equation}
Set 
$$
B = \left\{ \frac{T_x \psi}{v_\lambda(x)} ~  | ~  x \in \R^d  \right\}.
$$
We claim that $B$ is a bounded subset of $\mathcal{S}^{\{\omega\}}_{\{\V\}}$. Let $h, \mu >0$ be such that $\psi \in \mathcal{S}^{\omega,h}_{v_\mu}$. Condition $\{\condM\}$ yields that there are $ \nu \geq \lambda,\mu$ and $C > 0$ such that $v_\nu(y) \leq Cv_\mu(y-x) v_\lambda(x)$ for all $x,y \in \R^d$. Hence,
\begin{align*}
\sup_{x \in \R^d} \norm{\frac{T_x \psi}{v_\lambda(x)}}_{\mathcal{S}^{\omega,h}_{v_\nu} } &= \sup_{x\in \R^d}\sup_{\alpha \in \N^{d}} \sup_{ y \in \R^{d}} |\psi^{(\alpha)}(y-x)|\frac{v_\nu(y)}{v_\lambda(x)}\exp \left(-\frac{1}{h}\phi^*(h|\alpha|)\right) \\
&\leq C\norm{\psi}_{\mathcal{S}^{\omega,h}_{v_\mu} } < \infty.
\end{align*}
This shows that $B$ is a bounded subset of $\mathcal{S}^{\omega,h}_{v_\nu}$ and thus also of $\mathcal{S}^{\{\omega\}}_{\{\V\}}$. Next, by Lemma \ref{M12}$(i)$, there are $k,C > 0$ such that
$$
\frac{1}{k} \phi^*_\sigma (k(y+1)) + (\log\sqrt{d}) y \leq \phi^*_\sigma (y) 
+ \log C, \qquad y \geq 0.
$$
where $\phi^*_\sigma$ denotes the Young conjugate of the function $\phi_\sigma(x) = \sigma(e^x)$. Define
$$
p(\varphi) = \sup_{\alpha \in \N^d} \|\varphi^{(\alpha)}\|_{L^1} \exp \left(-\frac{1}{k}\phi^*_\sigma(k|\alpha|)\right) , \qquad \varphi \in \mathcal{S}^{\{\omega\}}_{\{\V\}}.
$$
Then, $p$ is a continuous seminorm on $\mathcal{S}^{\{\omega\}}_{\{\V\}}$. One can now show \eqref{Roumieu-todo} in the same way as \eqref{Beurling-todo} was shown in the Beurling case.
\end{proof}

\end{document}